[1994/12/01]
\documentclass{ijmart-mod}
\chardef\bslash=`\\ 

\usepackage{graphicx}
\usepackage[breaklinks=true]{hyperref}
\usepackage{mathtools}
\usepackage{caption}

\newtheorem{thm}{Theorem}[section]
\newtheorem{cor}[thm]{Corollary}
\newtheorem{lem}[thm]{Lemma}
\newtheorem{prop}[thm]{Proposition}

\theoremstyle{definition}

\newtheorem{rem}[thm]{Remark}

\theoremstyle{remark}

\newcommand{\eval}[2][\right]{\relax
  \ifx#1\right\relax \left.\fi#2#1\rvert}

\begin{document}
\title{Deep sections of the hypercube}

\author[L. Pournin]{Lionel Pournin}
\address{Universit{\'e} Paris 13, Villetaneuse, France}
\email{lionel.pournin@univ-paris13.fr}

\begin{abstract}
Consider a non-negative number $t$ and a hyperplane $H$ of $\mathbb{R}^d$ whose distance to the center of the hypercube $[0,1]^d$ is $t$. If $t$ is equal to $0$ and $H$ is orthogonal to a diagonal of $[0,1]^d$, it is known that the $(d-1)$\nobreakdash-dimensional volume of $H\cap[0,1]^d$ is a strictly increasing function of $d$ when $d$ is at least~$3$. The study of the monotonicity of this volume is extended for $t$ up to above $1/2$ and, when $d$ is large enough, for every non-negative $t$. In particular, a range for $t$ is identified such that this volume is a strictly decreasing function of $d$ over the positive integers. The local extremality of the $(d-1)$-dimensional volume of $H\cap[0,1]^d$ when $H$ is orthogonal to a diagonal of either $[0,1]^d$ or a lower dimensional face is also determined for the same values of $t$. It is shown for instance that when $t$ is above an explicit constant and $d$ is large enough, this volume is always strictly locally maximal when $H$ is orthogonal to a diagonal of $[0,1]^d$. A precise estimate for the convergence rate of the Eulerian numbers to their limit Gaussian behavior is provided along the way.
\end{abstract}
\maketitle


\section{Introduction}\label{DSH.sec.0}

Consider the $d$-dimensional unit hypercube $[0,1]^d$ and a hyperplane $H$ of $\mathbb{R}^d$ through the center of this hypercube. Denote by $V$ the $(d-1)$-dimensional volume of the intersection $H\cap[0,1]^d$. It is a result of Douglas Hensley \cite{Hensley1979} that the minimal value for $V$ is $1$ and that this only happens when $H$ is parallel to facet of $[0,1]^d$. In the same article, Hensley asked for the maximal value of $V$. This question was answered by Keith Ball~\cite{Ball1986} who showed that $V$ is at most $\sqrt{2}$ with equality precisely when $H$ is orthogonal to a diagonal of a square face of $[0,1]^d$. This problem can be generalized by considering a fixed non-negative number $t$ less than the circumradius of $[0,1]^d$ and a hyperplane $H$ of $\mathbb{R}^d$ whose distance to the center of $[0,1]^d$ is equal to $t$. In other words, $H$ is tangent to the sphere of radius $t$ that shares its center with the hypercube. Denoting again by $V$ the $(d-1)$-dimensional volume of $H\cap[0,1]^d$, Vitali Milman conjectured~\cite{KonigKoldobsky2011} that, if $t$ is at most $1/2$, then $V$ is minimal only when $H$ is orthogonal to a diagonal of a face of $[0,1]^d$. Milman also conjectured that $V$ is maximal, for most values of $t$, only when $H$ is orthogonal to a diagonal of a face of $[0,1]^d$. The latter conjecture is known to hold when
$$
\frac{\sqrt{d-2}}{2}<t<\frac{\sqrt{d}}{2}\mbox{,}
$$
and $d$ is at least $5$ \cite{MoodyStoneZachZvavitch2013,Pournin2023}. In that case, the maximal value of $V$ is achieved precisely when $H$ is orthogonal to a diagonal of $[0,1]^d$.

Local properties of $V$, thought as a function of $H$, have also been investigated \cite{AmbrusGargyan2024b,AmbrusGargyan2024,Konig2021,Pournin2024}. In particular it is known that, if $t$ is equal to $0$ and $d$ is at least $4$, then $V$ is always strictly locally maximal when $H$ is orthogonal to a diagonal of $[0,1]^d$ \cite{Pournin2024}. Moreover, also in the case when $t$ is equal to $0$, it is known that $V$ is never locally extremal when $H$ is orthogonal to a diagonal of a face of $[0,1]^d$ of dimension at least $3$ and less than $d$ \cite{AmbrusGargyan2024b} and that it has critical points when $H$ is not orthogonal to a diagonal of any face of the hypercube \cite{AmbrusGargyan2024}. If $t$ is within a range of $\sqrt{d}/\!\log d$ from the circumradius of the hypercube, then $V$ is always strictly locally maximal when $H$ is orthogonal to a diagonal of the hypercube \cite{Konig2021,Pournin2024} and if $t$ is within a range of $\sqrt{n}/\!\log n$ from the circumradius of $[0,1]^n$, then $V$ never locally extremal when $H$ is orthogonal to a diagonal of a $n$ dimensional face of $[0,1]^d$ \cite{Pournin2024,Pournin2025}. Volume sections of convex bodies other than the hypercube have also been considered \cite{Konig2021,Konig2023b,Konig2023a,LiuTkocz2020,MeyerPajor1988}. In particular, Alexandros Eskenazis, Piotr Nayar, and Tomasz Tkocz have shown that the result of~\cite{Ball1986} extends to the $\ell_p$ balls when $p$ is large enough \cite{EskenazisNayarTkocz2024}.

Denote by $I_d(t)$ the $(d-1)$-dimensional volume of $H\cap[0,1]^d$ in the case when $H$ is a hyperplane of $\mathbb{R}^d$ orthogonal to a diagonal of $[0,1]^d$ and whose distance to the center of $[0,1]^d$ is equal to $t$. It has been shown by Ferenc {\'A}goston Bartha, Ferenc Fodor, and Bernardo Gonz{\'a}lez Merino \cite{BarthaFodorGonzalezMerino2021} that $I_d(0)$ is a strictly increasing function of $d$ when $d$ is at least $3$. It is noteworthy that
$$
\lim_{d\rightarrow+\infty}I_d(0)=\sqrt{\frac{6}{\pi}}\mbox{.}
$$

In particular, this limit is barely less than the maximal $\sqrt{2}$. Here, the values of $t$ for which $I_d(t)$ is eventually strictly increasing or strictly decreasing when $d$ is large enough are identified via the following theorem.

\begin{thm}\label{DSH.sec.0.thm.1}
The limit of $d^2\bigl(I_{d+1}(t)-I_d(t)\bigr)$ as $d$ goes to infinity is
$$
\frac{3\sqrt{6}\bigl(1-24t^2+48t^4\bigr)}{20\sqrt{\pi}e^{6t^2}}\mbox{.}
$$
\end{thm}

The non-negative roots of $1-24t^2+48t^4$ are $\gamma^-$ and $\gamma^+$ where
\begin{equation}\label{DSH.sec.0.eq.0}
\gamma^\pm=\frac{1}{2}\sqrt{1\pm\sqrt{\frac{2}{3}}}\mbox{.}
\end{equation}

Therefore, $1-24t^2+48t^4$ is positive when $0\leq{t}<\gamma^-$ or $\gamma^+<t$ and negative when $\gamma^-<t<\gamma^+$. Hence, Theorem \ref{DSH.sec.0.thm.1} implies the following.

\begin{cor}\label{DSH.sec.0.cor.1}
Consider a non-negative number $t$ other than $\gamma^-$ and $\gamma^+$. For every large enough integer $d$, the difference $I_{d+1}(t)-I_d(t)$ is positive when $0\leq{t}<\gamma^-$ or $\gamma^+<t$ and negative when $\gamma^-<t<\gamma^+$.
\end{cor}

Theorem \ref{DSH.sec.0.thm.1} is established by extending the techniques used in \cite{BarthaFodorGonzalezMerino2021}. Further exploiting the bounds provided in \cite{BarthaFodorGonzalezMerino2021} on the first seven derivatives of a certain function allows to bound the lowest value of $d$ above which the monotonicity results stated by Corollary \ref{DSH.sec.0.cor.1} hold. This allows to analyse the sign of the difference $I_{d+1}(t)-I_d(t)$ down to $d$ equal to $1$ when $t$ is small enough or, equivalently when $H\cap[0,1]^d$ is a deep enough section of the hypercube. This can be done using symbolic computation by taking advantage of a piecewise polynomial expression for $I_d(t)$. All the values of $t$ within $[0,1/2]$ such that $I_d(t)$ is a strictly monotonic function of $d$ over the positive integers are identified. It is particularly noteworthy that there is a range for $t$ such that $I_d(t)$ is strictly decreasing. The bounds of the identified intervals are denoted $\alpha_{i,j}^-$ and $\alpha_{i,j}^\circ$ as they correspond values of $t$ such that $I_i(t)$ and $I_j(t)$ coincide.

\begin{thm}\label{DSH.sec.3.thm.3}
The two following statements hold.
\begin{enumerate}
\item[(i)] If $\alpha_{2,3}^-<t<\alpha_{3,4}^-$, then $I_d(t)$ is a strictly increasing function of $d$ over the positive integers where $\alpha_{3,4}^-$ is the solution of
$$
16-9\sqrt{3}-12\Bigl(8-3\sqrt{3}\Bigr)t^2+96t^3=0
$$
such that $0.144137<\alpha_{3,4}^-<0.144138$ and $\alpha_{2,3}^-$ is equal to
$$
\frac{2-\sqrt{31-12\sqrt{6}}}{6\sqrt{3}}
$$
and satisfies $0.0705012<\alpha_{2,3}^-<0.0705013$
\item[(ii)] If $\alpha_{3,4}^\circ<t\leq1/2$, then $I_d(t)$ is a strictly decreasing function of $d$ over the positive integers where $\alpha_{3,4}^\circ$ is the solution of
$$
32-27\sqrt{3}+108t-12\Bigl(16+3\sqrt{3}\Bigr)t^2+192t^3=0
$$
such that $0.407452<\alpha_{3,4}^\circ<0.407453$.
\end{enumerate}
\end{thm}

The two intervals identified by Theorem \ref{DSH.sec.3.thm.3} are exact in the sense that, if $t$ is less than $\alpha_{2,3}^-$, between $\alpha_{3,4}^-$ and $\alpha_{3,4}^\circ$, or greater than but close enough to $0.5$, then $I_d(t)$ is no longer strictly monotonic over the positive integers, but only above a certain value of $d$. In fact, it will also be shown that, if $t$ is at most $0.64607$ and not too close to $\gamma^-$, then $I_d(t)$ is a strictly monotonic function of $d$ when $d$ is at least $5$. The supremum and the infimum of $I_d(t)$ when $d$ ranges over the positive integers are further given for the same values of $t$. As a consequence, several intervals are determined such that, if $t$ belongs to these intervals, then the $(d-1)$\nobreakdash-dimensional volume of $H\cap[0,1]^d$ is not maximal when $H$ is orthogonal to a diagonal of any face of the hypercube.

\begin{thm}\label{DSH.sec.3.cor.2}
Consider a hyperplane $H$ of $\mathbb{R}^d$ whose distance to the center of $[0,1]^d$ is equal to $t$. If $0<t\leq\beta^-$, $\beta^+<t\leq0.20916$, or $\delta<t<0.5$ where $\beta^-$ and $\beta^+$ denote the two solutions of the equation
\begin{equation}\label{DSH.sec.3.cor.2.eq.1}
\frac{\sqrt{2-4t^2}-2t}{1-4t^2}=\sqrt{\frac{6}{\pi}}e^{-6t^2}
\end{equation}
satisfying $0.0181611<\beta^-<0.0181612$ and $0.165625<\beta^+<0.165626$ while $\delta$ denotes the root of the polynomial expression
$$
575\sqrt{5}-528\sqrt{6} -120\bigl(25\sqrt{5}-24\sqrt{6}\bigr)t^2+240\bigl(25\sqrt{5}-36\sqrt{6}\bigr)t^4+17280t^5
$$
that satisfies $0.222924<\delta<0.222925$, then the $(d-1)$-dimensional volume of $H\cap[0,1]^d$ cannot be maximal when $H$ is orthogonal to a diagonal of the hypercube $[0,1]^d$ or to a diagonal of any of its proper faces. 
\end{thm}

It has been shown by Georg P{\'o}lya that $I_d(t)$ converges to a Gaussian function as $d$ goes to infinity \cite{Polya1913}. On the way to proving the above theorems, the exact convergence rate is provided for this asymptotic estimate.

\begin{thm}\label{DSH.sec.0.thm.2}
If $d$ is large enough, then for every non-negative number $t$,
$$
\Biggl|I_d(t)-\sqrt{\frac{6}{\pi}}e^{-6t^2}\biggl(1-\frac{3-72t^2+144t^4}{20d}\biggr)\Biggr|<\frac{1}{d^{3/2}}\mbox{.}
$$
\end{thm}

It will be shown that the inequality stated by Theorem \ref{DSH.sec.0.thm.2} holds for all $d$ at least $136$ and in particular the convergence is uniform. As noticed by Richard Stanley \cite{Stanley1977} and Douglas Hensley \cite{Hensley1982}, $I_d(t)$ is closely related to the Eulerian numbers. These numbers can be arranged into a triangle like Pascal's triangle, and just as the binomial coefficients, they are asymptotically normal, as shown by Leonard Carlitz, David Kurtz, Richard Scoville, and Olaf Stackelberg~\cite{CarlitzKurtzScovilleStackelberg1972} who further bound the convergence order for this estimate. This bound will be made sharp as a consequence of Theorem \ref{DSH.sec.0.thm.2} (see Theorem \ref{DSH.sec.2.thm.1}).

Recall that the $(d-1)$-dimensional volume of $H\cap[0,1]^d$ is thought of as a function of $H$. Exploiting the results of \cite{Pournin2024,Pournin2025}, the same techniques allow to determine the local extremality of that volume when $H$ is orthogonal to a diagonal of a $n$-dimensional face of the hypercube $[0,1]^d$ for arbitrary values of $t$, provided $n$ is sufficiently large. This allows to prove, in particular, that for every $t$ greater than $\gamma^+$ and every large enough $d$, the $(d-1)$-dimensional volume of $H\cap[0,1]^d$ is always strictly locally maximal when $H$ is orthogonal to a diagonal of the hypercube and never extremal when $H$ is orthogonal to a diagonal of a sufficiently high-dimensional proper face. This confirms a conjecture of Hermann K{\"o}nig whereby down to a constant value of $t$, this volume is only maximal when $H$ is orthogonal to a diagonal of $[0,1]^d$. The local extremality of the $(d-1)$-dimensional volume of $H\cap[0,1]^d$ when $H$ is orthogonal to a diagonal of $[0,1]^d$ and $d$ is large enough is determined in general as follows.

\begin{thm}\label{DSH.sec.0.thm.3}
Consider a non-negative number $t$ and a hyperplane $H$ of $\mathbb{R}^d$ whose distance to the center of $[0,1]^d$ is equal to $t$. If $d$ is large enough then the $(d-1)$-dimensional volume of $H\cap[0,1]^d$ is
\begin{enumerate}
\item[(i)] strictly locally maximal when $t$ belongs to
$$
\Biggl[0,\gamma^--\frac{K}{d}\Biggr]\cup\Biggl[\gamma^++\frac{K}{d},\sqrt{\frac{\log d}{6}}\Biggr]
$$
and $H$ is orthogonal to a diagonal of $[0,1]^d$ and
\item[(ii)] strictly locally minimal when $t$ belongs to 
$$
\Biggl[\gamma^-+\frac{K}{d},\gamma^+-\frac{K}{d}\Biggr]
$$
and $H$ is orthogonal to a diagonal of $[0,1]^d$,
\end{enumerate}
where $K$ is a constant that can be taken equal to $76$.
\end{thm}


A similar result is provided when $H$ is orthogonal to a diagonal of a proper face of the hypercube $[0,1]^d$ of sufficiently large dimension.

\begin{thm}\label{DSH.sec.0.thm.3.5}
Consider a non-negative number $t$ and a hyperplane $H$ of $\mathbb{R}^d$ whose distance to the center of $[0,1]^d$ is equal to $t$. If $n$ is large enough but less than $d$, then the $(d-1)$-dimensional volume of $H\cap[0,1]^d$ is never locally extremal (even weakly so) when $t$ belongs to
$$
\Biggl[0,\gamma^--\frac{K}{n}\Biggr]\cup\Biggl[\gamma^-+\frac{K}{n},\gamma^+-\frac{K}{n}\Biggr]\cup\Biggl[\gamma^++\frac{K}{n},\sqrt{\frac{\log n}{6}}\Biggr]
$$
and $H$ is orthogonal to a diagonal of a $n$-dimensional face of $[0,1]^d$, where $K$ is a constant that can be taken equal to $76$.
\end{thm}

Again, the smallest values of $d$ and $n$ above which the results stated by Theorems \ref{DSH.sec.0.thm.3} and \ref{DSH.sec.0.thm.3.5} hold will be bounded by an explicit function of $t$. In the case when $t$ is small enough but not too close to $\gamma^-$ or to $\gamma^+$, this will allow to extend Theorem~\ref{DSH.sec.0.thm.3} down to when $d$ is equal to $4$ and Theorem \ref{DSH.sec.0.thm.3.5} down to when $n$ is equal to $4$ by using symbolic computation.

\begin{thm}\label{DSH.sec.0.thm.4}
Consider a non-negative number $t$. Let $H$ be a hyperplane of $\mathbb{R}^d$ whose distance to the center of $[0,1]^d$ is equal to $t$. If $d$ is at least $4$, then the $(d-1)$-dimensional volume of $H\cap[0,1]^d$ is
\begin{enumerate}
\item[(i)] strictly locally maximal when $t$ satisfies
\begin{equation}\label{DSH.sec.0.thm.4.eq.0}
t<\frac{7-\bigl(17-12\sqrt{2}\bigr)^{1/3}-\bigl(17+12\sqrt{2}\bigr)^{1/3}}{24}
\end{equation}
and $H$ is orthogonal to a diagonal of the hypercube $[0,1]^d$,
\item[(ii)] strictly locally minimal when $0.23593\leq{t}\leq0.59495$ and $H$ is orthogonal to a diagonal of the hypercube $[0,1]^d$, and
\item[(iii)] not locally extremal (even weakly so) when either $t$ satisfies (\ref{DSH.sec.0.thm.4.eq.0}) or belongs to $]1/4,1/2[\,\cup\,]1/2,0.59495]$ and $H$ is orthogonal to a diagonal of a proper face of dimension at least $4$ of the hypercube $[0,1]^d$.
\end{enumerate}
\end{thm}

It should be noted that (\ref{DSH.sec.0.thm.4.eq.0}) is a sharp inequality in the sense that if $t$ is larger, yet close enough to the right-hand side of (\ref{DSH.sec.0.thm.4.eq.0}), then the $(d-1)$-dimensional volume of $H\cap[0,1]^d$ is no longer locally maximal when $H$ is orthogonal to a diagonal of a $4$-dimensional face of the hypercube $[0,1]^d$ \cite{Pournin2024,Pournin2025}.

These results follow from two estimates for certain families of integrals. The first estimate is established in Section \ref{DSH.sec.1}. Theorem \ref{DSH.sec.0.thm.2} is proven in Section~\ref{DSH.sec.2} as a consequence and the corresponding bounds on the convergence rate of the Eulerian numbers are given. The announced results on the monotonicity of hypercube sections are then established in Section~\ref{DSH.sec.3} using the same estimate. The second estimate, that is less precise but holds for a more general family of integrals, is established in Section \ref{DSH.sec.4} and used in Section \ref{DSH.sec.5} in order to prove the announced results on the local extremality of hypercube sections.
 
\section{A second order estimate for a family of integrals}\label{DSH.sec.1}

The aim of the section to study the limit as $d$ goes to infinity of
\begin{equation}\label{DSH.sec.1.eq.0}
I_d(t)=\frac{2\sqrt{d}}{\pi}\!\int_{0}^{+\infty}\biggl(\frac{\sin\,u}{u}\biggr)^{\!d}\!\cos\Bigl(2\sqrt{d}tu\Bigr)du\mbox{.}
\end{equation}

More precisely, the following theorem is established.
\begin{thm}\label{DSH.sec.1.thm.0}
If $d$ is at least $136$ then, for every non-negative number $t$,
$$
\Biggl|I_d(t)-\sqrt{\frac{6}{\pi}}e^{-6t^2}\biggl(1+\frac{p(t)}{20d}+\frac{q(t)}{5600d^2}\biggr)\Biggr|<\frac{r(t)}{3d^3}
$$
where $p(t)$, $q(t)$, and $r(t)$ are defined as
\begin{equation}\label{DSH.sec.1.thm.0.eq.0}
\left\{
\begin{array}{l}
p(t)=-3+72t^2-144t^4\mbox{,}\\
q(t)=-65-6480t^2+96480t^4-246528t^6+145152t^8\mbox{,}\\
r(t)=2+3t+2t^2+3t^3\mbox{.}
\end{array}
\right.
\end{equation}
\end{thm}

The proof of Theorem \ref{DSH.sec.1.thm.0} relies on the change of variables
$$
\frac{\sin\,u}{u}=e^{-x^2/6}
$$
already used in \cite{BarthaFodorGonzalezMerino2021} in the case when $t$ is equal to $0$. This change of variables can equivalently be written as $x=\phi(u)$ where
\begin{equation}\label{DSH.sec.1.eq.0.1}
\phi(u)=\sqrt{-6\log\biggl(\frac{\sin u}{u}\biggr)}\mbox{.}
\end{equation}

While it is only valid when $\sin u$ is positive, this will not be a problem because, if one splits the integral in the right-hand side of (\ref{DSH.sec.1.eq.0}) at a fixed number $\varepsilon$ contained in the interval $]0,\pi/2[$, then the portion of this integral over $[\varepsilon,+\infty[$ goes exponentially fast to $0$ as $d$ goes to infinity. A slightly more general statement is proven that will be useful in Section \ref{DSH.sec.4}.

\begin{lem}\label{DSH.sec.1.lem.0}
Consider a non-negative integer $k$. If $0<\varepsilon<\pi/2$, then for every integer $d$ at least $k+2$ and every non-negative number $t$,
$$
\Biggl|\int_\varepsilon^{+\infty}\biggl(\frac{\sin\,u}{u}\biggr)^{\!d}\!\cos\Bigl(2\sqrt{d}tu\Bigr)u^kdu\Biggr|\leq\pi\biggl(\frac{\pi}{2}\biggr)^{\!k}\biggl(\max\biggl\{\frac{\sin\,\varepsilon}{\varepsilon},\frac{2}{\pi}\biggr\}\biggr)^{\!d}\mbox{.}
$$
\end{lem}
\begin{proof}
Consider a number $\varepsilon$ contained in the interval $]0,\pi/2[$, an integer $d$ at least $k+2$, and a non-negative number $t$. Observe that
\begin{multline*}
\biggl|\int_\varepsilon^{+\infty}\biggl(\frac{\sin\,u}{u}\biggr)^{\!d}\!\cos\Bigl(2\sqrt{d}tu\Bigr)u^kdu\biggr|\leq\biggl(\frac{\pi}{2}\biggr)^k\!\int_\varepsilon^{\pi/2}\biggl(\frac{\sin\,u}{u}\biggr)^{\!d}du\\
\hfill+\int_{\pi/2}^{+\infty}\frac{du}{u^{d-k}}\mbox{.}
\end{multline*}

Further observe that $1/u^{d-k}$ can be explicitly integrated. As in addition, $(\sin u)/u$ is a decreasing function of $u$ on $[0,\pi/2]$, it follows that
\begin{multline*}
\biggl|\int_\varepsilon^{+\infty}\biggl(\frac{\sin\,u}{u}\biggr)^{\!d}\!\cos\Bigl(2\sqrt{d}tu\Bigr)u^kdu\biggr|\leq\biggl(\frac{\pi}{2}-\varepsilon\biggr)\biggl(\frac{\pi}{2}\biggr)^k\biggl(\frac{\sin\,\varepsilon}{\varepsilon}\biggr)^{\!d}\\
\hfill+\frac{1}{(d-k-1)}\biggl(\frac{2}{\pi}\biggr)^{\!d-k-1}
\end{multline*}
and further bounding each term in the right-hand side yields
$$
\biggl|\int_\varepsilon^{+\infty}\biggl(\frac{\sin\,u}{u}\biggr)^{\!d}\!\cos\Bigl(2\sqrt{d}tu\Bigr)u^kdu\biggr|\leq\biggl(\frac{\pi}{2}\biggr)^{\!k+1}\Biggl(\biggl(\frac{\sin\,\varepsilon}{\varepsilon}\biggr)^{\!d}+\biggl(\frac{2}{\pi}\biggr)^{\!d}\Biggr)
$$
which implies the desired inequality.
\end{proof}

As observed in \cite{BarthaFodorGonzalezMerino2021}, $\phi$ is a strictly increasing analytic function of $u$ within the interval $[0,1]$. Moreover $\phi(0)$ is equal to $0$. Hence, by the Lagrange inversion theorem, $\phi$ admits an inverse which is also a strictly increasing analytic function on the interval $[0,\phi(1)]$. In the sequel, the inverse of $\phi$ is considered a function of $x$ and it is denoted by $\psi$ in order not to overburden notations.

Now observe that $\phi(1)$ is greater than $1$. As a consequence, $\psi$ is also an analytic function of $x$ on the interval $[0,1]$. Denote 
\begin{equation}\label{DSH.sec.1.eq.4}
F_d(t)=\frac{2\sqrt{d}}{\pi}\!\int_{0}^{\psi(1)}\biggl(\frac{\sin\,u}{u}\biggr)^{\!d}\!\cos\Bigl(2\sqrt{d}tu\Bigr)du\mbox{.}
\end{equation}

Further note that $0.9832<\psi(1)<0.9833$. Therefore,
\begin{equation}\label{DSH.sec.1.eq.4.5}
\frac{2}{\pi}<\frac{\sin\,\psi(1)}{\psi(1)}<\frac{17}{20}\mbox{.}
\end{equation}

In particular, taking $\varepsilon$ equal to $\psi(1)$ in the statement of Lemma \ref{DSH.sec.1.lem.0} and setting $k$ to zero immediately implies the following.

\begin{lem}\label{DSH.sec.1.lem.1}
For every integer $d$ at least $2$ and every non-negative number $t$,
$$
\bigl|I_d(t)-F_d(t)\bigr|\leq2\sqrt{d}\biggl(\frac{17}{20}\biggr)^{\!d}\mbox{.}
$$
\end{lem}
%
%

By the change of variables $u=\psi(x)$, one can rewrite (\ref{DSH.sec.1.eq.4}) as
\begin{equation}\label{DSH.sec.1.eq.5}
F_d(t)=\frac{2\sqrt{d}}{\pi}\!\int_0^1e^{-dx^2\!/6}\cos\Bigl(2\sqrt{d}t\psi(x)\Bigr)\frac{d\psi}{dx}(x)dx\mbox{.}
\end{equation}

The next step is to estimate the integrand in the right-hand side of (\ref{DSH.sec.1.eq.5}). Since $\psi$ is analytic on the interval $[0,1]$, this function admits a Taylor series expansion around $0$ whose first terms are given in \cite{BarthaFodorGonzalezMerino2021}. In particular,
\begin{equation}\label{DSH.sec.1.eq.6}
\psi(x)=x-\frac{x^3}{60}-\frac{13x^5}{151200}+O(x^7)
\end{equation}
as $x$ goes to $0$. Moreover, the absolute value of each of the derivatives of $\psi$ admits a maximum over the interval $[0,1]$. The following bounds on these maxima can be immediately derived from Table D1 in \cite{BarthaFodorGonzalezMerino2021}.

\begin{prop}[\cite{BarthaFodorGonzalezMerino2021}]\label{DSH.sec.3.prop.1}
For every number $x$ in $[0,1]$,
\begin{enumerate}
\item[(i)] $\displaystyle\biggl|\frac{d\psi}{dx}(x)\biggr|<1.00001$,\\
\item[(ii)] $\displaystyle\biggl|\frac{d^3\psi}{dx^3}(x)\biggr|<0.10718$,\\
\item[(iii)] $\displaystyle\biggl|\frac{d^5\psi}{dx^5}(x)\biggr|<0.02591$, and\\
\item[(iv)] $\displaystyle\biggl|\frac{d^7\psi}{dx^7}(x)\biggr|<0.79461$.
\end{enumerate}
\end{prop}

The cosinus factor of the integrand in the right-hand side of (\ref{DSH.sec.1.eq.5}) will be taken care of using the trigonometric identity
\begin{multline}\label{DSH.sec.1.eq.7.25}
\cos\Bigl(2\sqrt{d}t\psi(x)\Bigr)=\cos\Bigl(2\sqrt{d}tx\Bigr)\cos\Bigl(2\sqrt{d}t\bigl(\psi(x)-x\bigr)\Bigr)\\
\hfill-\sin\Bigl(2\sqrt{d}tx\Bigr)\sin\Bigl(2\sqrt{d}t\bigl(\psi(x)-x\bigr)\Bigr)\mbox{.}
\end{multline}

Using this identity in the integrand yields
\begin{multline}\label{DSH.sec.1.eq.7.5}
F_d(t)=\frac{2\sqrt{d}}{\pi}\!\int_0^1e^{-dx^2\!/6}\cos\Bigl(2\sqrt{d}tx\Bigr)c_d(x,t)dx\\
\hfill-\frac{2\sqrt{d}}{\pi}\!\int_0^1e^{-dx^2\!/6}\sin\Bigl(2\sqrt{d}tx\Bigr)s_d(x,t)dx
\end{multline}
where
\begin{equation}\label{DSH.sec.1.eq.7.6}
\left\{
\begin{array}{l}
\displaystyle{c_d(x,t)=\cos\Bigl(2\sqrt{d}t\bigl(\psi(x)-x\bigr)\Bigr)\frac{d\psi}{dx}(x)}\mbox{,}\\[\bigskipamount]
\displaystyle{s_d(x,t)=\sin\Bigl(2\sqrt{d}t\bigl(\psi(x)-x\bigr)\Bigr)\frac{d\psi}{dx}(x)}\mbox{.}\\
\end{array}
\right.
\end{equation}

The two terms in the right-hand side of (\ref{DSH.sec.1.eq.7.5}) will be estimated separately. The following lemma provides an estimate for $c_d(x,t)$.

\begin{lem}\label{DSH.sec.1.lem.3}
For every positive integer $d$, every non-negative number $t$, and every number $x$ contained in the interval $[0,1]$,
$$
\bigg|c_d(x,t)-1+\frac{x^2}{20}+\frac{13x^4}{30240}+\frac{dt^2x^6}{1800}\biggr|\leq\frac{7x^6}{6250}+\frac{dt^2x^8}{20000}+\frac{d\sqrt{d}t^3x^9}{125000}\mbox{.}
$$
\end{lem}
\begin{proof}
Recall that $\psi$ is an analytic function of $x$ on $[0,1]$. The first seven terms of its Taylor series expansion around $0$ are given by (\ref{DSH.sec.1.eq.6}). By order $2$ and $6$ Taylor series approximations of the derivative of $\psi$ around $0$ with a Lagrange form remainder, one obtains that for any number $x$ in $[0,1]$,
\begin{equation}\label{DSH.sec.1.lem.3.eq.0}
\frac{d\psi}{dx}(x)=1+\frac{d^3\psi}{dx^3}(\xi_2)\frac{x^2}{2}\mbox{.}
\end{equation}
where $0\leq\xi_2\leq{x}$ and
\begin{equation}\label{DSH.sec.1.lem.3.eq.0.5}
\frac{d\psi}{dx}(x)=1-\frac{x^2}{20}-\frac{13x^4}{30240}+\frac{d^7\psi}{dx^7}(\xi_6)\frac{x^6}{6!}\mbox{.}
\end{equation}
where $0\leq\xi_6\leq{x}$. Likewise, an order $3$ Taylor series approximation of $\cos z$ around $0$ with a Lagrange form remainder yields
$$
\cos z=1-\frac{z^2}{2}+\frac{z^3}{6}\sin\zeta
$$
where $\zeta$ depends on $z$. Using the change of variables
$$
z=2\sqrt{d}t\bigl(\psi(x)-x\bigr)\mbox{,}
$$
the latter equality can be rewritten into
$$
\cos\Bigl(2\sqrt{d}t\bigl(\psi(x)-x\bigr)\Bigr)=1-2dt^2\bigl(\psi(x)-x\bigr)^2+\frac{4}{3}d\sqrt{d}t^3\bigl(\psi(x)-x\bigr)^3\sin\zeta
$$
where $\zeta$ depends on $d$, $t$, and $x$. Multiplying this with the derivative of $\psi$ at $x$ and using (\ref{DSH.sec.1.lem.3.eq.0}) and (\ref{DSH.sec.1.lem.3.eq.0.5}) in the resulting expression yields
\begin{multline}\label{DSH.sec.1.lem.3.eq.2}
c_d(x,t)=1-\frac{x^2}{20}-\frac{13x^4}{30240}+\frac{d^7\psi}{dx^7}(\xi_6)\frac{x^6}{6!}-2dt^2\bigl(\psi(x)-x\bigr)^2\\
\hfill-dt^2\bigl(\psi(x)-x\bigr)^2\frac{d^3\psi}{dx^3}(\xi_2)x^2+\frac{4}{3}d\sqrt{d}t^3\bigl(\psi(x)-x\bigr)^3\frac{d\psi}{dx}(x)\sin\zeta\mbox{.}
\end{multline}

In turn, one obtains from order $3$ and $5$ Taylor series approximations of $\psi$ around $0$ with a Lagrange form remainder that for every $x$ in $[0,1]$,
$$
\psi(x)=x+\frac{d^3\psi}{dx^3}(\eta_3)\frac{x^3}{6}
$$
where $0\leq\eta_3\leq{x}$ and
$$
\psi(x)=x-\frac{x^3}{60}+\frac{d^5\psi}{dx^5}(\eta_5)\frac{x^5}{5!}
$$
where $0\leq\eta_5\leq{x}$. Substituting these expressions in (\ref{DSH.sec.1.lem.3.eq.2}) yields
\begin{multline*}
c_d(x,t)=1-\frac{x^2}{20}-\frac{13x^4}{30240}+\frac{d^7\psi}{dx^7}(\xi_6)\frac{x^6}{6!}-2dt^2\biggl(\frac{x^3}{60}-\frac{d^5\psi}{dx^5}(\eta_5)\frac{x^5}{5!}\biggr)^{\!2}\\
\hfill-dt^2\biggl(\frac{d^3\psi}{dx^3}(\eta_3)\biggr)^{\!2}\frac{d^3\psi}{dx^3}(\xi_2)\frac{x^8}{36}+d\sqrt{d}t^3\biggl(\frac{d^3\psi}{dx^3}(\eta_3)\biggr)^{\!3}\frac{d\psi}{dx}(x)\frac{x^9}{162}\sin\zeta\mbox{.}
\end{multline*}

However, observe that
$$
\biggl(\frac{x^3}{60}-\frac{d^5\psi}{dx^5}(\eta_5)\frac{x^5}{5!}\biggr)^{\!2}=\frac{x^6}{3600}-\frac{d^5\psi}{dx^5}(\eta_5)\frac{x^8}{3600}+\biggl(\frac{d^5\psi}{dx^5}(\eta_5)\biggr)^{\!2}\frac{x^{10}}{14400}\mbox{.}
$$

As a consequence,
\begin{multline*}
\biggl|c_d(x,t)-1+\frac{x^2}{20}+\frac{13x^4}{30240}+\frac{dt^2x^6}{1800}\biggr|\leq\biggl|\frac{d^7\psi}{dx^7}(\xi_6)\biggr|\frac{x^6}{6!}\\
+\biggl|\frac{d^5\psi}{dx^5}(\eta_5)\biggr|\frac{dt^2x^8}{1800}+\biggl|\frac{d^5\psi}{dx^5}(\eta_5)\biggr|^2\frac{dt^2x^{10}}{7200}+\biggl|\frac{d^3\psi}{dx^3}(\eta_3)\biggr|^2\biggl|\frac{d^3\psi}{dx^3}(\xi_2)\biggr|\frac{dt^2x^8}{36}\\
\hfill+\biggl|\frac{d^3\psi}{dx^3}(\eta_3)\biggr|^3\biggl|\frac{d\psi}{dx}(x)\biggr|\frac{d\sqrt{d}t^3x^9}{162}\mbox{.}
\end{multline*}

As $0\leq{x}\leq1$, one can further upper bound $x^{10}$ by $x^8$ in the right-hand side of this inequality. Finally, it follows from Proposition \ref{DSH.sec.3.prop.1} that the coefficients of $x^6$, $dt^2x^8$, and $d\sqrt{d}t^3x^9$ in the resulting expression are less than $7/6250$, $1/20000$, and $1/125000$, respectively, which completes the proof.
\end{proof}

The function $s_d(x,t)$ is now estimated as follows. 
\begin{lem}\label{DSH.sec.1.lem.4}
For every positive integer $d$, every non-negative number $t$, and every number $x$ contained in the interval $[0,1]$,
$$
\biggl|s_d(x,t)+\frac{\sqrt{d}t}{30}\biggl(x^3-\frac{113x^5}{2520}\biggr)\biggr|\leq\frac{19\sqrt{d}tx^7}{50000}+\frac{d\sqrt{d}t^3x^9}{125000}\mbox{.}
$$
\end{lem}
\begin{proof}
Recall that $\psi$ is analytic on $[0,1]$ and that the first coefficients of its Taylor series expansion around $0$ are given by (\ref{DSH.sec.1.eq.6}). In particular, by an order $4$ approximation of its derivative with Lagrange form remainder,
\begin{equation}\label{DSH.sec.1.lem.4.eq.0}
\frac{d\psi}{dx}(x)=1-\frac{x^2}{20}+\frac{d^5\psi}{dx^5}(\xi_4)\frac{x^4}{4!}
\end{equation}
for every number $x$ contained in the interval $[0,1]$, where $0\leq\xi_4\leq{x}$. Similarly, it follows from an order $3$ Taylor series approximation of $\sin z$ around $0$ with Lagrange form remainder that, for every real number $z$,
$$
\sin z=z-\frac{z^3}{6}\cos\zeta\mbox{,}
$$
where $\zeta$ is a number that depends on $z$. Substituting
$$
z=2\sqrt{d}t\bigl(\psi(x)-x\bigr)\mbox{,}
$$
in the latter equality yields
$$
\sin\Bigl(2\sqrt{d}t\bigl(\psi(x)-x\bigr)\Bigr)=2\sqrt{d}t\bigl(\psi(x)-x\bigr)-\frac{4}{3}d\sqrt{d}t^3\bigl(\psi(x)-x\bigr)^3\cos\zeta
$$
where $\zeta$ depends on $d$, $t$, and $x$. One obtains, by multiplying this equality with the derivative of $\psi$ at $x$ and by using (\ref{DSH.sec.1.lem.4.eq.0}), that
\begin{multline}\label{DSH.sec.1.lem.4.eq.1}
s_d(x,t)=2\sqrt{d}t\bigl(\psi(x)-x\bigr)-\sqrt{d}t\bigl(\psi(x)-x\bigr)\frac{x^2}{10}\\
+\sqrt{d}t\bigl(\psi(x)-x\bigr)\frac{d^5\psi}{dx^5}(\xi_4)\frac{x^4}{12}-\frac{4}{3}d\sqrt{d}t^3\bigl(\psi(x)-x\bigr)^3\frac{d\psi}{dx}(x)\cos\zeta\mbox{.}
\end{multline}

One obtains from order $3$, $5$, and $7$ Taylor series approximations of $\psi$ around $0$ with Lagrange form remainder that, for every number $x$ in $[0,1]$,
$$
\psi(x)=x+\frac{d^3\psi}{dx^3}(\eta_3)\frac{x^3}{6}
$$
where $0\leq\eta_3\leq{x}$,
$$
\psi(x)=x-\frac{x^3}{60}+\frac{d^5\psi}{dx^5}(\eta_5)\frac{x^5}{5!}
$$
where $0\leq\eta_5\leq{x}$, and
$$
\psi(x)=x-\frac{x^3}{60}-\frac{13x^5}{151200}+\frac{d^7\psi}{dx^7}(\eta_7)\frac{x^7}{7!}
$$
where $0\leq\eta_7\leq{x}$. Note that each of these equalities results in an expression for $\psi(x)-x$. Substituting these expressions in (\ref{DSH.sec.1.lem.4.eq.1}) yields
\begin{multline*}
s_d(x,t)=-\frac{\sqrt{d}t}{30}\biggl(x^3+\frac{13x^5}{2520}-\frac{d^7\psi}{dx^7}(\eta_7)\frac{x^7}{84}\biggr)\\
+\sqrt{d}t\biggl(\frac{x^3}{60}-\frac{d^5\psi}{dx^5}(\eta_5)\frac{x^5}{5!}\biggr)\frac{x^2}{10}+\sqrt{d}t\frac{d^3\psi}{dx^3}(\eta_3)\frac{d^5\psi}{dx^5}(\xi_4)\frac{x^7}{72}\\
\hfill-d\sqrt{d}t^3\biggl(\frac{d^3\psi}{dx^3}(\eta_3)\biggr)^{\!3}\frac{d\psi}{dx}(x)\frac{x^9}{162}\cos\zeta
\end{multline*}
and as a consequence,
\begin{multline}\label{DSH.sec.1.lem.4.eq.2}
\Biggl|s_d(x,t)+\frac{\sqrt{d}t}{30}\biggl(x^3-\frac{113x^5}{2520}\biggr)\Biggr|\leq\biggl|\frac{d^7\psi}{dx^7}(\eta_7)\biggr|\frac{\sqrt{d}tx^7}{2520}+\biggl|\frac{d^5\psi}{dx^5}(\eta_5)\biggr|\frac{\sqrt{d}tx^7}{1200}\\
\hfill+\biggl|\frac{d^3\psi}{dx^3}(\eta_3)\frac{d^5\psi}{dx^5}(\xi_4)\biggr|\frac{\sqrt{d}tx^7}{72}+\biggl|\frac{d^3\psi}{dx^3}(\eta_3)\biggr|^3\biggl|\frac{d\psi}{dx}(x)\biggr|\frac{d\sqrt{d}t^3x^9}{162}\mbox{.}
\end{multline}

Finally observe that, by Proposition \ref{DSH.sec.3.prop.1}, the coefficients of $\sqrt{d}tx^7$ and $d\sqrt{d}t^3x^9$ in the right-hand side of this inequality are less than $19/50000$ and $1/125000$, respectively. The desired bound immediately follows.
\end{proof}

For every non-negative number $t$ and every non-negative integer $k$, denote
\begin{equation}\label{DSH.sec.1.eq.6.7}
T_k(t)=\frac{2}{\pi}\!\int_{0}^{+\infty}e^{-y^2\!/6}\cos\bigl(2ty\bigr)y^kdy
\end{equation}
when $k$ even and 
$$
T_k(t)=\frac{2}{\pi}\!\int_{0}^{+\infty}e^{-y^2\!/6}\sin\bigl(2ty\bigr)y^kdy
$$
when $k$ is odd. One can see by Leibniz' integral rule that $T_0$ is a continuously differentiable function of $t$. Differentiating under the integral sign shows that this function satisfies the differential equation
$$
\frac{dT_0}{dt}(t)=-12T_0(t)\mbox{.}
$$

Since $T_0(0)$ is a Gaussian integral, solving this equation yields
\begin{equation}\label{DSH.sec.1.eq.7}
T_0(t)=\sqrt{\frac{6}{\pi}}e^{-6t^2}\mbox{.}
\end{equation}

When $k$ is positive, $T_k(t)$ can be computed recursively as follows.
\begin{prop}\label{DSH.sec.1.prop.1}
For any non-negative integer $k$,
$$
T_{k+2}(t)=3(k+1)T_k(t)+6t(-1)^{k+1}T_{k+1}(t)\mbox{.}
$$

Moreover, $T_1(t)$ is equal to $6tT_0(t)$.
\end{prop}
\begin{proof}
This follows from an integration by parts.
\end{proof}

One obtains from Proposition \ref{DSH.sec.1.prop.1} that $T_k(t)$ is the product of $T_0(t)$ with a degree $k$ polynomial of $t$. In particular, for the first few integers $k$,
\begin{equation}\label{DSH.sec.1.eq.8}
\left\{
\begin{array}{l}
T_2(t)=3\bigl(1-12t^2\bigr)T_0(t)\mbox{,}\\
T_3(t)=54t\bigl(1-4t^2\bigr)T_0(t)\mbox{,}\\
T_4(t)=27\bigl(1-24t^2+48t^4\bigr)T_0(t)\mbox{,}\\
T_5(t)=162t\bigl(5-40t^2+48t^4\bigr)T_0(t)\mbox{,}\\
T_6(t)=81\bigl(5-180t^2+720t^4-576t^6\bigl)T_0(t)\mbox{.}\\
\end{array}
\right.
\end{equation}

These quantities will appear in the proof of Theorem \ref{DSH.sec.1.thm.0}. The following proposition bounds another quantity that will appear in this proof.

\begin{prop}\label{DSH.sec.1.prop.2}
If $k$ is non-negative and $d$ at least $\sqrt{6k}$, then
$$
\int_{\sqrt{d}}^{+\infty}e^{-y^2\!/6}y^kdy\leq\frac{6\sqrt{d}^{k-1}}{e^{d/6}}\mbox{.}
$$
\end{prop}
\begin{proof}
Assume that $k$ is non-negative and observe that $e^{-y^2\!/12}y^{k-1}$ is a differentiable function of $y$ on $\bigl[\sqrt{6k},+\infty\bigr[$. One obtains by differentiation that this function is decreasing on that interval. Hence, if $d$ is at least $\sqrt{6k}$,
$$
\int_{\sqrt{d}}^{+\infty}e^{-y^2\!/6}y^kdy\leq\frac{\sqrt{d}^{k-1}}{e^{d/12}}\!\int_{\sqrt{d}}^{+\infty}e^{-y^2\!/12}ydy\mbox{.}
$$

The integral in the right-hand side of this inequality can be explicitly computed and doing so provides the desired upper bound.
\end{proof}

Theorem \ref{DSH.sec.1.thm.0} is now established as a consequence of Lemmas \ref{DSH.sec.1.lem.3} and \ref{DSH.sec.1.lem.4}.

\begin{proof}[Proof of Theorem \ref{DSH.sec.1.thm.0}]
First denote
\begin{multline}\label{DSH.sec.1.thm.0.eq.1}
J_d(t)=\frac{2\sqrt{d}}{\pi}\int_0^1e^{-dx^2\!/6}\cos\Bigl(2\sqrt{d}tx\Bigr)\biggl(1-\frac{x^2}{20}-\frac{13x^4}{30240}-\frac{dt^2x^6}{1800}\biggr)dx\\
\hfill+\frac{dt}{15\pi}\int_0^1e^{-dx^2\!/6}\sin\Bigl(2\sqrt{d}tx\Bigr)\biggl(x^3-\frac{113x^5}{2520}\biggr)dx\mbox{.}
\end{multline}

It follows from (\ref{DSH.sec.1.eq.7.5}) and the triangle inequality that
\begin{multline*}
\bigl|F_d(t)-J_d(t)\bigr|\leq\frac{2\sqrt{d}}{\pi}\!\int_0^1e^{-dx^2\!/6}\biggl|c_d(x,t)-1+\frac{x^2}{20}+\frac{13x^4}{30240}+\frac{dt^2x^6}{1800}\biggr|dx\\
\hfill+\frac{2\sqrt{d}}{\pi}\!\int_0^1e^{-dx^2\!/6}\Biggl|s_d(x,t)+\frac{\sqrt{d}t}{30}\biggl(x^3-\frac{113x^5}{2520}\biggr)\Biggr|dx\mbox{.}
\end{multline*}

Hence, according to Lemmas \ref{DSH.sec.1.lem.3} and \ref{DSH.sec.1.lem.4},
\begin{multline}\label{DSH.sec.1.thm.0.eq.1.25}
\bigl|F_d(t)-J_d(t)\bigr|\leq\frac{7\sqrt{d}}{3125\pi}\!\int_0^1e^{-dx^2\!/6}x^6dx+\frac{19dt}{25000\pi}\!\int_0^1e^{-dx^2\!/6}x^7dx\\
\hfill+\frac{d\sqrt{d}t^2}{10000\pi}\!\int_0^1e^{-dx^2\!/6}x^8dx+\frac{d^2t^3}{31250\pi}\int_0^1e^{-dx^2\!/6}x^9dx
\end{multline}
for every positive integer $d$ and every non-negative number $t$. Now observe that the integrals in the right-hand side of (\ref{DSH.sec.1.thm.0.eq.1.25}) can be bounded by an explicit constant. Indeed, by the change of variables $y=\sqrt{d}x$,
\begin{equation}\label{DSH.sec.1.thm.0.eq.1.375}
\begin{array}{rcl}
\displaystyle\sqrt{d}\!\int_0^1e^{-dx^2\!/6}x^kdx & \!\!\!\!=\!\!\!\! & \displaystyle\frac{1}{\sqrt{d}^k}\!\int_0^{\sqrt{d}}e^{-y^2\!/6}y^kdy\mbox{,}\\[\bigskipamount]
& \!\!\!\!\leq\!\!\!\! & \displaystyle\frac{1}{\sqrt{d}^k}\!\int_0^{+\infty}e^{-y^2\!/6}y^kdy\mbox{.}\\
\end{array}
\end{equation}

However, integrations by parts and the Gaussian integral yield
\begin{equation}\label{DSH.sec.1.thm.0.eq.1.4}
\int_0^{+\infty}e^{-y^2\!/6}y^kdy=
\left\{
\begin{array}{l}
\displaystyle\frac{k!\sqrt{\pi}}{(k/2)!}\sqrt{3/2}^{k+1}\mbox{ when }k\mbox{ is even,}\\[\bigskipamount]
\displaystyle3\bigl((k-1)/2\bigr)!\sqrt{6}^{k-1}\mbox{ when }k\mbox{ is odd.}\\
\end{array}
\right.
\end{equation}

It follows from (\ref{DSH.sec.1.thm.0.eq.1.25}), (\ref{DSH.sec.1.thm.0.eq.1.375}), and (\ref{DSH.sec.1.thm.0.eq.1.4}) that
\begin{equation}\label{DSH.sec.1.thm.0.eq.1.5}
\bigl|F_d(t)-J_d(t)\bigr|\leq\frac{567}{625d^3}\sqrt{\frac{3}{2\pi}}+\frac{9234t}{3125\pi{d^3}}+\frac{1701t^2}{2000d^3}\sqrt{\frac{3}{2\pi}}+\frac{46656t^3}{15625\pi{d^3}}\mbox{.}
\end{equation}

Using the change of variables $y=\sqrt{d}x$ again in (\ref{DSH.sec.1.thm.0.eq.1}) yields
\begin{multline*}
J_d(t)=\frac{2}{\pi}\!\int_0^{\sqrt{d}}e^{-y^2\!/6}\cos\Bigl(2ty\Bigr)\biggl(1-\frac{y^2}{20d}-\frac{13y^4}{30240d^2}-\frac{t^2y^6}{1800d^2}\biggr)dy\\
\hfill+\frac{t}{15\pi}\!\int_0^{\sqrt{d}}e^{-y^2\!/6}\sin\Bigl(2ty\Bigr)\biggl(\frac{y^3}{d}-\frac{113y^5}{2520d^2}\biggr)dy\mbox{.}
\end{multline*}

Therefore by the triangle inequality,
\begin{multline}\label{DSH.sec.1.thm.0.eq.2}
\biggl|J_d(t)-T_0(t)+\frac{T_2(t)}{20d}+\frac{13T_4(t)}{30240d^2}+\frac{t^2T_6(t)}{1800d^2}-\frac{tT_3(t)}{30d}+\frac{113tT_5(t)}{75600d^2}\biggr|\\
\hfill\leq\frac{2}{\pi}\!\int_{\sqrt{d}}^{+\infty}e^{-y^2\!/6}\biggl(1+\frac{y^2}{20d}+\frac{ty^3}{30d}+\frac{13y^4}{30240d^2}+\frac{113ty^5}{75600d^2}+\frac{t^2y^6}{1800d^2}\biggr)dy\mbox{.}
\end{multline}

According to (\ref{DSH.sec.1.eq.7}) and (\ref{DSH.sec.1.eq.8}), 
\begin{multline*}
T_0(t)-\frac{T_2(t)}{20d}-\frac{13T_4(t)}{30240d^2}-\frac{t^2T_6(t)}{1800d^2}+\frac{tT_3(t)}{30d}-\frac{113tT_5(t)}{75600d^2}\\
\hfill=\sqrt{\frac{6}{\pi}}e^{-6t^2}\biggl(1+\frac{p(t)}{20d}+\frac{q(t)}{5600d^2}\biggr)
\end{multline*}
where
$$
\left\{
\begin{array}{l}
p(t)=-3+72t^2-144t^4\mbox{,}\\
q(t)=-65-6480t^2+96480t^4-246528t^6+145152t^8\mbox{.}
\end{array}
\right.
$$

Further note that, according to Proposition \ref{DSH.sec.1.prop.2},
\begin{multline*}
\frac{2}{\pi}\!\int_{\sqrt{d}}^{+\infty}e^{-y^2\!/6}\biggl(1+\frac{y^2}{20d}+\frac{ty^3}{30d}+\frac{13y^4}{30240d^2}+\frac{113ty^5}{75600d^2}+\frac{t^2y^6}{1800d^2}\biggr)dy\\
\hfill\leq\frac{12\sqrt{d}}{\pi{e^{d/6}}}\biggl(\frac{1}{d}+\frac{1}{20d}+\frac{t}{30\sqrt{d}}+\frac{13}{30240}+\frac{113t}{75600\sqrt{d}}+\frac{t^2}{1800}\biggr)
\end{multline*}
and that, when $d$ is at least $15$,
$$
\frac{1}{d}+\frac{1}{20d}+\frac{t}{30\sqrt{d}}+\frac{13}{30240}+\frac{113t}{75600\sqrt{d}}+\frac{t^2}{1800}\leq\frac{1+t+t^2}{12}\mbox{.}
$$

As a consequence, it follows from (\ref{DSH.sec.1.thm.0.eq.2}) that, when $d$ is at least $15$,
$$
\Biggl|J_d(t)-\sqrt{\frac{6}{\pi}}e^{-6t^2}\biggl(1+\frac{p(t)}{20d}+\frac{q(t)}{5600d^2}\biggr)\Biggr|\leq\frac{\sqrt{d}(1+t+t^2)}{\pi{e^{d/6}}}\mbox{.}
$$

Observe that the coefficient of $1+t+t^2$ in the right-hand side is at most $1/(60d^3)$ when $d$ is at least $118$. Hence, for any such value of $d$
\begin{equation}\label{DSH.sec.1.thm.0.eq.3}
\Biggl|J_d(t)-\sqrt{\frac{6}{\pi}}e^{-6t^2}\biggl(1+\frac{p(t)}{20d}+\frac{q(t)}{5600d^2}\biggr)\Biggr|\leq\frac{1+t+t^2}{60d^3}\mbox{.}
\end{equation}

Finally, observe that when $d$ is at least $136$,
$$
2\sqrt{d}\biggl(\frac{17}{20}\biggr)^{\!d}<\frac{1}{60d^3}\mbox{.}
$$

Hence, it follows from Lemma~\ref{DSH.sec.1.lem.1} that for any such value of $d$,
\begin{equation}\label{DSH.sec.1.thm.0.eq.4}
\bigl|I_d(t)-F_d(t)\bigr|\leq\frac{1}{60d^3}\mbox{.}
\end{equation}

One obtains by combining~(\ref{DSH.sec.1.thm.0.eq.1.5}), (\ref{DSH.sec.1.thm.0.eq.3}), and (\ref{DSH.sec.1.thm.0.eq.4}) with the triangle inequality that, for any integer $d$ at least $136$ and any non-negative number $t$,
\begin{multline*}
\Biggl|I_d(t)-\sqrt{\frac{6}{\pi}}e^{-6t^2}\biggl(1+\frac{p(t)}{20d}+\frac{q(t)}{5600d^2}\biggr)\Biggr|\leq\biggl(\frac{567}{625}\sqrt{\frac{3}{2\pi}}+\frac{1}{30}\biggr)\frac{1}{d^3}\\
\hfill+\biggl(\frac{9234}{3125\pi}+\frac{1}{60}\biggr)\frac{t}{d^3}+\biggl(\frac{1701}{2000}\sqrt{\frac{3}{2\pi}}+\frac{1}{60}\biggr)\frac{t^2}{d^3}+\frac{46656t^3}{15625\pi{d^3}}\mbox{.}
\end{multline*}

As the coefficients of $1/d^3$, $t/d^3$, $t^2/d^3$, and $t^3/d^3$ in the right-hand side are less than $2/3$, $1$, $2/3$, and $1$, respectively, the result follows.
\end{proof}

\section{The asymptotic normality of volume sections}\label{DSH.sec.2}

It has been shown by Georg P{\'o}lya \cite[Equation (11)]{Polya1913} that
$$
\lim_{d\rightarrow+\infty}I_d(t)=\sqrt{\frac{6}{\pi}}e^{-6t^2}\mbox{.}
$$

In other words, the volume of the intersection between a unit hypercube and a hyperplane orthogonal to a diagonal of that hypercube is asymptotically normal. In this section, a sharp bound on the convergence order of this estimate is given and a similar bound is derived for the asymptotic normality of Eulerian numbers. As a first step, the polynomial function $q$ that appears in the statement of Theorem~\ref{DSH.sec.0.thm.2} is bounded as follows. While the proof of this statement is elementary, it is provided for completeness.

\begin{prop}\label{DSH.sec.3.prop.2}
For every non-negative number $t$,
$$
\frac{\bigl|65+6480t^2-96480t^4+246528t^6-145152t^8\bigr|}{e^{6t^2}}\leq2281\mbox{.}
$$
\end{prop}
\begin{proof}
First consider the function $f$ of $t$ defined over $[0,+\infty[$ by
\begin{equation}\label{DSH.sec.3.prop.1.eq.0}
f(t)=\frac{65+6480t^2-96480t^4}{e^{6t^2}}\mbox{.}
\end{equation}

Observe that $f$ goes to $0$ as $t$ goes to infinity. Moreover, $f(0)$ is equal to $65$. As $f$ is differentiable, its absolute value is maximal either when $t$ is equal to $0$ or when its derivative vanishes. However,
$$
\frac{df}{dt}(t)=\frac{60t\bigl(203-7728t^2+19296t^4\bigr)}{e^{6t^2}}\mbox{.}
$$

Hence, the positive values of $t$ at which the derivative of $f$ vanishes can be explicitly obtained by solving a quadratic equation for $t^2$. Substituting these values in  (\ref{DSH.sec.3.prop.1.eq.0}), shows that, for every non-negative number $t$,
\begin{equation}\label{DSH.sec.3.prop.1.eq.1}
|f(t)|<1170\mbox{.}
\end{equation}

Now denote
$$
g(t)=\frac{246528t^6-145152t^8}{e^{6t^2}}\mbox{.}
$$

Again, $g$ is a differentiable function of $t$ on $[0,+\infty[$ that goes to $0$ as $t$ goes to infinity. Moreover it vanishes when $t$ is equal to $0$ and
$$
\frac{dg}{dt}(t)=\frac{13824t^5\bigl(107-298t^2+126t^4\bigr)}{e^{6t^2}}\mbox{.}
$$

Again, the positive values of $t$ for which the derivative of $g$ is equal to $0$ can be explicitly obtained by solving a quadratic equation for $t^2$ and computing $g$ at these values shows that, for every non-negative number $t$,
\begin{equation}\label{DSH.sec.3.prop.1.eq.2}
|g(t)|<1111\mbox{.}
\end{equation}

The proposition follows from (\ref{DSH.sec.3.prop.1.eq.1}), (\ref{DSH.sec.3.prop.1.eq.2}), and the triangle inequality.
\end{proof}

Theorem~\ref{DSH.sec.0.thm.2}, is now derived from Theorem \ref{DSH.sec.1.thm.0}. In fact, Theorem~\ref{DSH.sec.0.thm.2} is an immediate consequence of the following more precise statement.

\begin{thm}\label{DSH.sec.2.thm.0.5}
If $d$ is at least $136$ then, for every non-negative number $t$,
$$
\Biggl|I_d(t)-\sqrt{\frac{6}{\pi}}e^{-6t^2}\biggl(1-\frac{3-72t^2+144t^4}{20d}\biggr)\Biggr|<\frac{1}{d^{3/2}}\mbox{.}
$$
\end{thm}
\begin{proof}
Consider an integer $d$ at least $136$ and a non-negative number $t$. It follows from Theorem \ref{DSH.sec.1.thm.0} and Proposition~\ref{DSH.sec.3.prop.2} that
\begin{equation}\label{DSH.sec.2.thm.0.5.eq.1}
\Biggl|I_d(t)-\sqrt{\frac{6}{\pi}}e^{-6t^2}\biggl(1-\frac{3-72t^2+144t^4}{20d}\biggr)\Biggr|\leq\frac{r(t)}{3d^3}+\sqrt{\frac{6}{\pi}}\frac{2281}{5600d^2}
\end{equation}
where $r$ is the polynomial function of $t$ defined by (\ref{DSH.sec.1.thm.0.eq.0}). Two cases are now considered depending on the value of $t$. First assume that $t$ is at most $\sqrt{d}/2$. In this case, it follows from the expression of $r$ that
$$
\frac{r(t)}{3d^3}+\sqrt{\frac{6}{\pi}}\frac{2281}{5600d^2}\leq\frac{2}{3d^3}+\frac{1}{2d^{5/2}}+\biggl(\frac{1}{6}+\sqrt{\frac{6}{\pi}}\frac{2281}{5600}\biggr)\frac{1}{d^2}+\frac{1}{8d^{3/2}}\mbox{.}
$$

Since $d$ is at least $136$, the right-hand side of this inequality is less than $1/d^{3/2}$ and the theorem therefore follows from (\ref{DSH.sec.2.thm.0.5.eq.1}) in this case. Now assume that $t$ is greater than $\sqrt{d}/2$. In that case, $t$ is greater than the circumradius of $[0,1]^d$ and the corresponding hypercube section is empty. Therefore, $I_d(t)$ is equal to $0$ (see also the discussion in \cite{KonigKoldobsky2011}). Further observe that $e^{-6t^2}$ and
$$
\frac{3-72t^2+144t^4}{e^{6t^2}}
$$
are positive, decreasing functions of $t$ on the interval $[1,+\infty[$ (this can be checked easily by differentiation). Hence, as $t$ is greater than $\sqrt{d}/2$,
\begin{equation}\label{DSH.sec.2.thm.0.5.eq.2}
\Biggl|I_d(t)-\sqrt{\frac{6}{\pi}}e^{-6t^2}\biggl(1-\frac{3-72t^2+144t^4}{20d}\biggr)\Biggr|\leq\sqrt{\frac{6}{\pi}}\frac{3+2d+9d^2}{20de^{3d/2}}\mbox{.}
\end{equation}

However, since $d$ is a positive integer,
$$
\sqrt{\frac{6}{\pi}}\frac{3+2d+9d^2}{20de^{3d/2}}\leq\sqrt{\frac{6}{\pi}}\frac{7d}{10e^{3d/2}}
$$

Note that the right-hand side of this inequality is less than $1/d^{3/2}$ for all positive values of $d$. Hence, the theorem follows from (\ref{DSH.sec.2.thm.0.5.eq.2}).
\end{proof}

The remainder of the section is devoted to deriving from Theorem \ref{DSH.sec.2.thm.0.5} a precise estimate on the convergence rate of the Eulerian numbers to a Gaussian function. Recall that the Eulerian numbers $A(d,i)$ form a triangle of numbers similar to Pascal's triangle. They are defined for all pairs of positive integers $d$ and $i$ such that $i$ is less than $d$ by the recurrence relation
$$
A(d,i)=(d-i)A(d-1,i-1)+(i+1)A(d-1,i)
$$
where $A(d,1)$ and $A(d,d-1)$ are both equal to $1$ for every integer $d$ greater than $1$. Eulerian numbers happen to be proportional to certain volume sections of the hypercube. Indeed, if $H$ is a hyperplane of $\mathbb{R}^d$ orthogonal to a diagonal of the hypercube $[0,1]^d$ and that contains at least one vertex of $[0,1]^d$, then the intersection of $H$ and $[0,1]^d$ is a polytope called \emph{a hypersimplex} that arises in algebraic combinatorics \cite{GabrielovGelfandLosik1975,GelfandGoreskyMacPhersonSerganova1987}. The constraint that $H$ contains at least one vertex of the hypercube $[0,1]^d$ can be written in terms of its distance $t$ to the center of this hypercube: $H$ contains a vertex of $[0,1]^d$ if and only if
\begin{equation}\label{DSH.sec.2.eq.1}
t=\frac{\sqrt{d}}{2}-\frac{i}{\sqrt{d}}
\end{equation}
where $i$ is an integer satisfying $0\leq{i}\leq{d/2}$. Recall that $I_d(t)$ denotes the $(d-1)$-dimensional volume of $H\cap[0,1]^d$ when $H$ is orthogonal to a diagonal of $[0,1]^d$. It is remarked in \cite{Hensley1982,Stanley1977} that this volume can be expressed as
$$
I_d(t)=\frac{\sqrt{d}}{(d-1)!}A(d,i)
$$
when $t$ satisfies (\ref{DSH.sec.2.eq.1}) or equivalently, when $H\cap[0,1]^d$ is a hypersimplex. It has been shown in \cite{CarlitzKurtzScovilleStackelberg1972} that the Eulerian numbers and therefore the volumes of the hypersimplices are asymptotically normal in the following sense.
\begin{thm}[\cite{CarlitzKurtzScovilleStackelberg1972}]\label{DSH.sec.2.thm.0}
There exist two numbers $K$ and $N$ such that, for every integer $d$ greater than $N$ and every non-negative number $t$ less than $\sqrt{d}/2$,
$$
\left|\frac{\sqrt{d}}{(d-1)!}A\biggl(d,\biggl\lceil\frac{d}{2}-\sqrt{d}t\biggr\rceil\biggr)-\sqrt{\frac{6}{\pi}}e^{-6t^2}\right|\leq\frac{K}{d^{1/4}}\mbox{.}
$$
\end{thm}

The convergence order can be estimated more precisely as a consequence of Theorem \ref{DSH.sec.2.thm.0.5}. In order to provide this estimate, the following technical statement will be needed whose proof is provided for completeness.

\begin{prop}\label{DSH.sec.2.prop.1}
If $d$ is at least $2$, then for every non-negative number $t$,
\begin{equation}\label{DSH.sec.2.prop.1.eq.1}
\frac{e^{12t/\sqrt{d}}-1}{e^{6t^2}}\leq\sqrt{\frac{6}{d}}e^{12/\sqrt{d}}
\end{equation}
\end{prop}
\begin{proof}
Consider a non-negative number $t$ and denote
$$
f(t)=\frac{e^{12t/\sqrt{d}}-1}{e^{6t^2}}\mbox{.}
$$

By a Taylor series approximation around $0$ with Lagrange form remainder,
$$
e^{12t/\sqrt{d}}=1+e^\xi\frac{12t}{\sqrt{d}}
$$
where $0\leq\xi\leq12t/\sqrt{d}$. Assume that $t$ is at most $1$. In that case, $\xi$ is at most $12/\sqrt{d}$. As in addition, $e^{6t^2}$ is at least $1+6t^2$, it follows that
$$
f(t)\leq\frac{e^{12/\sqrt{d}}}{\sqrt{d}}\frac{12t}{1+6t^2}
$$

Since the right-hand side of this inequality, considered a function of $t$, is maximal when $t$ is equal to $1/\sqrt{6}$, this shows that (\ref{DSH.sec.2.prop.1.eq.1}) holds when $t$ is at most $1$. Now assume that $t$ is greater than $1$ and observe that
$$
\frac{df}{dt}(t)=12\frac{e^{12t/\sqrt{d}}(1-\sqrt{d}t)+\sqrt{d}t}{\sqrt{d}e^{6t^2}}
$$

As $t$ is greater than $1$, the term $1-\sqrt{d}t$ in the right-hand side is non-positive when $d$ is positive and lower bounding $e^{12t/\sqrt{d}}$ by $1+12t/\sqrt{d}$ yields
$$
\frac{df}{dt}(t)\leq12\frac{\sqrt{d}+12t-12\sqrt{d}t^2}{de^{6t^2}}\mbox{.}
$$

The numerator of in the right-hand side is a quadratic function of $t$ whose roots are both less than $1$ when $d$ is at least $2$. Hence, this numerator is negative for any such value of $d$ and $f$ is a decreasing function of $t$ on $[1,+\infty[$. It follows that $f(t)$ is upper bounded by $f(1)$. Since, by the above, $f(1)$ is at most the right-hand side of (\ref{DSH.sec.2.prop.1.eq.1}), this completes the proof.
\end{proof}

Theorem \ref{DSH.sec.2.thm.0} is now strengthened as a consequence of Theorem \ref{DSH.sec.0.thm.2}.

\begin{thm}\label{DSH.sec.2.thm.1}
There exist an integer $N$ such that, for every integer $d$ greater than $N$ and every non-negative number $t$ less than $\sqrt{d}/2$,
\begin{equation}\label{DSH.sec.2.thm.1.eq.0}
\left|\frac{\sqrt{d}}{(d-1)!}A\biggl(d,\biggl\lceil\frac{d}{2}-\sqrt{d}t\biggr\rceil\biggr)-\sqrt{\frac{6}{\pi}}e^{-6t^2}\right|\leq\frac{7e^{12/\sqrt{d}}}{\sqrt{\pi{d}}}\mbox{.}
\end{equation}
\end{thm}
\begin{proof}
Consider a non-negative number $t$ less than $\sqrt{d}/2$. Denote
$$
\overline{t}=\frac{\sqrt{d}}{2}-\frac{1}{\sqrt{d}}\biggl\lceil\frac{d}{2}-\sqrt{d}t\biggr\rceil
$$
and observe that
$$
\overline{t}\leq{t}<\overline{t}+\frac{1}{\sqrt{d}}\mbox{.}
$$

As a consequence,
$$
\begin{array}{rcl}
\displaystyle\left|\sqrt{\frac{6}{\pi}}e^{-6\overline{t}^2}-\sqrt{\frac{6}{\pi}}e^{-6t^2}\right| & \!\!\!\!=\!\!\!\! & \displaystyle\sqrt{\frac{6}{\pi}}\frac{e^{6(t^2-\overline{t}^2)}-1}{e^{6t^2}}\mbox{,}\\[\bigskipamount]
& \!\!\!\!=\!\!\!\! & \displaystyle\sqrt{\frac{6}{\pi}}\frac{e^{6(t+\overline{t})(t-\overline{t})}-1}{e^{6t^2}}\mbox{,}\\[\bigskipamount]
& \!\!\!\!\leq\!\!\!\! & \displaystyle\sqrt{\frac{6}{\pi}}\frac{e^{12t/\sqrt{d}}-1}{e^{6t^2}}\mbox{.}\\
\end{array}
$$

Therefore, according to Proposition \ref{DSH.sec.2.prop.1},
\begin{equation}\label{DSH.sec.2.thm.1.eq.1}
\left|\sqrt{\frac{6}{\pi}}e^{-6\overline{t}^2}-\sqrt{\frac{6}{\pi}}e^{-6t^2}\right|\leq\frac{6e^{12/\sqrt{d}}}{\sqrt{\pi{d}}}
\end{equation}
when $d$ is at least $2$. Now recall that
$$
I_d(\overline{t})=\frac{\sqrt{d}}{(d-1)!}A\biggl(d,\biggl\lceil\frac{d}{2}-\sqrt{d}t\biggr\rceil\biggr)\mbox{.}
$$

Moreover, observe that
$$
\frac{3-72t^2+144t^4}{e^{6t^2}}
$$
is a bounded function of $t$. Therefore, by Theorem \ref{DSH.sec.2.thm.0.5}, 
\begin{equation}\label{DSH.sec.2.thm.1.eq.2}
\left|\frac{\sqrt{d}}{(d-1)!}A\biggl(d,\biggl\lceil\frac{d}{2}-\sqrt{d}t\biggr\rceil\biggr)-\sqrt{\frac{6}{\pi}}e^{-6\overline{t}^2}\right|<\frac{K}{d}\mbox{.}
\end{equation}
when $d$ is at least $136$, where $K$ is a number that does not depend on $t$ or $d$. In particular, there exists an integer $N$ at least $2$, that does not depend on $t$ or $d$ such that when $d$ is greater than $N$,
$$
\frac{K}{d}\leq\frac{e^{12/\sqrt{d}}}{\sqrt{\pi{d}}}
$$

The theorem is obtained as a consequence from (\ref{DSH.sec.2.thm.1.eq.1}) and (\ref{DSH.sec.2.thm.1.eq.2}).
\end{proof}

The estimate provided by Theorem \ref{DSH.sec.2.thm.1} on the convergence rate of the Eulerian numbers to a Gaussian function is quite precise. Indeed, denote
\begin{equation}\label{DSH.sec.2.eq.2}
i=\biggl\lceil\frac{d}{2}-0.3\sqrt{d}\biggr\rceil
\end{equation}
and, assuming that $d$ is at least $45$, consider the interval
\begin{equation}\label{DSH.sec.2.eq.2.5}
\biggl]\frac{\sqrt{d}}{2}-\frac{i+1}{\sqrt{d}},\frac{\sqrt{d}}{2}-\frac{i}{\sqrt{d}}\biggr]\mbox{.}
\end{equation}

For any number $t$ in that interval,
$$
A\biggl(d,\biggl\lceil\frac{d}{2}-\sqrt{d}t\biggr\rceil\biggr)=A(d,i)\mbox{.}
$$

Hence, the left-hand side of (\ref{DSH.sec.2.thm.1.eq.0}) cannot be uniformly less than
\begin{equation}\label{DSH.sec.2.eq.3}
\frac{1}{2}\Biggl|\sqrt{\frac{6}{\pi}}e^{-6\left(\frac{\sqrt{d}}{2}-\frac{i+1}{\sqrt{d}}\right)^{\!2}}-\sqrt{\frac{6}{\pi}}e^{-6\left(\frac{\sqrt{d}}{2}-\frac{i}{\sqrt{d}}\right)^{\!2}}\Biggr|
\end{equation}
in that interval. However, observe that according to (\ref{DSH.sec.2.eq.2}), 
$$
0.3-\frac{2}{\sqrt{d}}<\frac{\sqrt{d}}{2}-\frac{i+1}{\sqrt{d}}<\frac{\sqrt{d}}{2}-\frac{i}{\sqrt{d}}\leq0.3\mbox{.}
$$

Since $d$ is at least $45$, it follows in particular that the two bounds of (\ref{DSH.sec.2.eq.2.5}) are contained in the interval $\bigl[0,\sqrt{d}/2\bigr[$. Moreover, (\ref{DSH.sec.2.eq.3}) is at least
$$
\sqrt{\frac{3}{2\pi}}e^{-0.54}\left|e^{\frac{3.6}{\sqrt{d}}-\frac{24}{d}}-1\right|\geq\sqrt{\frac{3}{2\pi}}e^{-0.54}\biggl(\frac{3.6}{\sqrt{d}}-\frac{24}{d}\biggr)\mbox{.}
$$

It follows that as $d$ goes to infinity, the left-hand side of (\ref{DSH.sec.2.thm.1.eq.0}) cannot be uniformly bounded above away from $1.4496/\sqrt{d}$. However, the coefficient of $1/\sqrt{d}$ in the right-hand side of (\ref{DSH.sec.2.thm.1.eq.0}) goes to less than $3.9494$ as $d$ goes to infinity. As a consequence, the gap for the above estimate of the convergence rate of the Eulerian numbers to their limit Gaussian function is less than $2.5$.

\section{The monotonicity of diagonal sections}\label{DSH.sec.3}

In this section, the sign of $I_{d+1}(t)-I_d(t)$ is studied both as $d$ goes to infinity while $t$ is fixed and for all $d$ when $t$ is close to $0$. Theorem \ref{DSH.sec.0.thm.1} treats the former case by providing an exact asymptotic estimate for this quantity as $d$ goes to infinity. It is obtained as a consequence of Theorem \ref{DSH.sec.1.thm.0}.

\begin{proof}[Proof of Theorem \ref{DSH.sec.0.thm.1}]
According to Theorem \ref{DSH.sec.1.thm.0},
\begin{equation}\label{DSH.sec.0.thm.1.eq.1}
\Biggl|I_d(t)-\sqrt{\frac{6}{\pi}}e^{-6t^2}\biggl(1+\frac{p(t)}{20d}+\frac{q(t)}{5600d^2}\biggr)\Biggr|<\frac{r(t)}{3d^3}
\end{equation}
for every integer $d$ at least $136$ and every non-negative number $t$, where $q$ and $r$ are degree $8$ and $3$ polynomial functions of $t$ and
\begin{equation}\label{DSH.sec.0.thm.1.eq.2}
p(t)=-3\bigl(1-24t^2+48t^4\bigr)\mbox{.}
\end{equation}

Assume that $d$ is at least $136$. In that case, (\ref{DSH.sec.0.thm.1.eq.1}) still holds when $d$ is replaced by $d+1$ and it then follows from the triangle inequality that
$$
\Biggl|I_{d+1}(t)-I_d(t)+\sqrt{\frac{6}{\pi}}e^{-6t^2}\biggl(\frac{p(t)}{20d(d+1)}+\frac{(1+2d)q(t)}{5600d^2(d+1)^2}\biggr)\Biggr|<\frac{2r(t)}{3d^3}
$$
and therefore that
$$
\Biggl|d^2\bigl(I_{d+1}(t)-I_d(t)\bigr)+\sqrt{\frac{6}{\pi}}\frac{dp(t)e^{-6t^2}}{20(d+1)}\Biggr|<\sqrt{\frac{6}{\pi}}\frac{(1+2d)|q(t)|e^{-6t^2}}{5600(d+1)^2}+\frac{2r(t)}{3d}
$$

Observe that $(1+2d)/5600$ is less than $(d+1)/20$ . Moreover,
$$
\Biggl|\sqrt{\frac{6}{\pi}}\frac{p(t)}{20e^{6t^2}}-\sqrt{\frac{6}{\pi}}\frac{dp(t)e^{-6t^2}}{20(d+1)}\Biggr|=\sqrt{\frac{6}{\pi}}\frac{|p(t)|e^{-6t^2}}{20(d+1)}\mbox{.}
$$

Hence, by the triangle inequality,
$$
\Biggl|d^2\bigl(I_{d+1}(t)-I_d(t)\bigr)+\sqrt{\frac{6}{\pi}}\frac{p(t)}{20e^{6t^2}}\Biggr|<\sqrt{\frac{6}{\pi}}\frac{|p(t)|+|q(t)|}{20(d+1)e^{6t^2}}+\frac{2r(t)}{3d}\mbox{.}
$$

As a consequence,
$$
\lim_{d\rightarrow\infty}d^2\bigl(I_{d+1}(t)-I_d(t)\bigr)=-\sqrt{\frac{6}{\pi}}\frac{p(t)}{e^{6t^2}}
$$
which, by (\ref{DSH.sec.0.thm.1.eq.2}), completes the proof. 
\end{proof}

It turns out that Theorem \ref{DSH.sec.1.thm.0} further allows to bound the convergence rate of $d^2\bigl(I_{d+1}(t)-I_d(t)\bigr)$ to its limit as $d$ goes to infinity and, therefore the sign of $I_{d+1}(t)-I_d(t)$ when $d$ is large enough. Recall that $\gamma^-$ and $\gamma^+$ denote the two non-negative roots of $1-24t^2+48t^4$ with the convention that $\gamma^-$ is less than $\gamma^+$. An expression for these roots is given by (\ref{DSH.sec.0.eq.0}). An explicit threshold $\Delta(t)$ is given by the following theorem for every non-negative number $t$ other than $\gamma^-$ and $\gamma^+$, such that if $d$ is at least $\Delta(t)$, then  $I_{d+1}(t)-I_d(t)$ and the limit of $d^2\bigl(I_{d+1}(t)-I_d(t)\bigr)$ as $d$ goes to infinity have the same sign.

\begin{thm}\label{DSH.sec.3.thm.1}
Consider a non-negative number $t$ other than $\gamma^-$ and $\gamma^+$. If $d$ is an integer at least $\Delta(t)$, then $I_{d+1}(t)-I_d(t)$ is positive when $0\leq{t}<\gamma^-$ or $\gamma^+<t$ and negative when $\gamma^-<t<\gamma^+$, where
%
\begin{equation}\label{DSH.sec.3.thm.1.eq.0}
\Delta(t)=\max\biggl\{136,e^{6t^2}\frac{12+10t+7t^2+10t^3}{|1-24t^2+48t^4|}\biggr\}\mbox{.}
\end{equation}
%
\end{thm}
\begin{proof}
Assume that $d$ is at least $136$. According to Theorem \ref{DSH.sec.1.thm.0},
\begin{equation}\label{DSH.sec.3.thm.1.eq.1}
\Biggl|I_d(t)-\sqrt{\frac{6}{\pi}}e^{-6t^2}\biggl(1+\frac{p(t)}{20d}+\frac{q(t)}{5600d^2}\biggr)\Biggr|<\frac{r(t)}{3d^3}\mbox{,}
\end{equation}
where $p$, $q$, and $r$ are polynomial functions of $t$ given by~(\ref{DSH.sec.1.thm.0.eq.0}). Hence, writing (\ref{DSH.sec.3.thm.1.eq.1}) for both $I_d(t)$ and $I_{d+1}(t)$ and using the triangle inequality yields
\begin{multline*}
\Biggl|I_{d+1}(t)-I_d(t)+\sqrt{\frac{6}{\pi}}e^{-6t^2}\biggl(\frac{p(t)}{20d(d+1)}+\frac{(1+2d)q(t)}{5600d^2(d+1)^2}\biggr)\Biggr|\\
\hfill<\frac{r(t)}{3}\biggl(\frac{1}{d^3}+\frac{1}{(d+1)^3}\biggr)
\end{multline*}
and in turn,
\begin{multline}\label{DSH.sec.3.thm.1.eq.2}
\Biggl|I_{d+1}(t)-I_d(t)+\sqrt{\frac{6}{\pi}}\frac{p(t)e^{-6t^2}}{20d(d+1)}\Biggr|<\sqrt{\frac{6}{\pi}}\frac{(1+2d)|q(t)|e^{-6t^2}}{5600d^2(d+1)^2}\\
\hfill+\frac{r(t)}{3}\biggl(\frac{1}{d^3}+\frac{1}{(d+1)^3}\biggr)\mbox{.}
\end{multline}

Note that $1+2d$ is less than $2+2d$ and that, according to Proposition \ref{DSH.sec.3.prop.2}, $|q(t)|e^{-6t^2}$ is less than $2281$. This allows to upper bound the first term in the right-hand side of (\ref{DSH.sec.3.thm.1.eq.2}). The second term in the right-hand side of this inequality can be upper bounded by observing that $r(t)$ is positive and that
$$
\frac{1}{d^3}+\frac{1}{(d+1)^3}\leq\frac{2}{d^2(d+1)}\mbox{.}
$$

One obtains as a consequence that
$$
\Biggl|I_{d+1}(t)-I_d(t)+\sqrt{\frac{6}{\pi}}\frac{p(t)e^{-6t^2}}{20d(d+1)}\Biggr|<\sqrt{\frac{6}{\pi}}\frac{2281}{2800d^2(d+1)}+\frac{2r(t)}{3d^2(d+1)}\mbox{.}
$$

It follows from this inequality that if
\begin{equation}\label{DSH.sec.3.thm.1.eq.3}
\sqrt{\frac{6}{\pi}}\frac{|p(t)|e^{-6t^2}}{20d(d+1)}\geq\sqrt{\frac{6}{\pi}}\frac{2281}{2800d^2(d+1)}+\frac{2r(t)}{3d^2(d+1)}
\end{equation}
then $I_{d+1}(t)-I_d(t)$ and $p(t)$ must have opposite signs. Note that multiplying (\ref{DSH.sec.3.thm.1.eq.3}) by an adequate quantity, it can be rewritten into
$$
d\geq\frac{e^{6t^2}}{|p(t)|}\biggl(\frac{2281}{140}+\sqrt{\frac{\pi}{6}}\frac{40r(t)}{3}\biggr)\mbox{.}
$$

However, according to (\ref{DSH.sec.1.thm.0.eq.0}), $|p(t)|=3\bigl|1-24t^2+48t^4\bigr|$ and
$$
\frac{2281}{140}+\sqrt{\frac{\pi}{6}}\frac{40r(t)}{3}=\frac{2281}{140}+\frac{80}{3}\sqrt{\frac{\pi}{6}}+40\sqrt{\frac{\pi}{6}}t+\frac{80}{3}\sqrt{\frac{\pi}{6}}t^2+40\sqrt{\frac{\pi}{6}}t^3\mbox{.}
$$

Observe that the coefficients of $t$, $t^2$, and $t^3$ in the right-hand side of this equality are less than $30$, $21$, and $30$, respectively, while the constant term is less than $36$. As a consequence, the inequality
\begin{equation}\label{DSH.sec.3.thm.1.eq.4}
d\geq\max\biggl\{136,e^{6t^2}\frac{12+10t+7t^2+10t^3}{|1-24t^2+48t^4|}\biggr\}\mbox{,}
\end{equation}
implies (\ref{DSH.sec.3.thm.1.eq.3}). In particular, if $d$ satisfies (\ref{DSH.sec.3.thm.1.eq.4}), then $I_{d+1}(t)-I_d(t)$ and $p(t)$ have opposite signs. Since $p(t)$ is negative when $0\leq{t}<\gamma^-$ or $\gamma^+<t$ and negative when $\gamma^-<t<\gamma^+$, this completes the proof.
\end{proof}

The function $\Delta$ defined by (\ref{DSH.sec.3.thm.1.eq.0}) provides an explicit upper bound on the smallest dimension above which $I_d(t)$ becomes a monotonic function of $d$. Note that $\Delta(t)$ goes to infinity as $t$ goes to $\gamma^-$ or $\gamma^+$. In particular, when $t$ is close to either of these two values, determining the exact variations of  $I_d(t)$ as a function of $d$ over $[1,\Delta(d)]\cap\mathbb{N}$ might be difficult. However, if $t$ is sufficiently far away from these two values, yet small enough (because $\Delta(t)$ goes exponentially fast to infinity as $t$ grows large), $\Delta(t)$ gets below a few hundreds and it is possible to recover these variations. In order to do so, it is convenient to use an alternative expression for $I_d(t)$. According for instance to \cite[Theorem 2.2]{Pournin2023},
\begin{equation}\label{DSH.sec.3.thm.2.eq.1}
I_d(t)=\frac{\sqrt{d}}{(d-1)!}\sum_{i=0}^{\lfloor{z}\rfloor}(-1)^i{d\choose{i}}(z-i)^{d-1}
\end{equation}
where
$$
z=\frac{d}{2}-\sqrt{d}t\mbox{.}
$$

It immediately follows from this expression that $I_d(t)$ is polynomial function of $t$, whose degree less than $d$, on each interval of the form
$$
U_{d,i}=\biggl[\frac{\sqrt{d}}{2}-\frac{i}{\sqrt{d}},\frac{\sqrt{d}}{2}-\frac{i+1}{\sqrt{d}}\biggr]
$$
where $i$ is a non-negative integer less than $d/2$. In particular, it is possible to determine whether $I_{d+1}(t)-I_d(t)$ is positive or negative within an interval of the form $U_{d,i}\cap{U_{d,j}}$ by computing its sign at one bound of $U_{d,i}\cap{U_{d,j}}$ and by checking whether the corresponding polynomial has a root in $U_{d,i}\cap{U_{d,j}}$. In addition, it is possible to bound the value of the roots contained in $U_{d,i}\cap{U_{d,j}}$ by using symbolic computation in order to get precise ranges for $t$ such that $I_{d-1}(t)-I_d(t)$ is negative or positive. Since there are only finitely-many intervals of the form $U_{d,i}\cap{U_{d,j}}$ for each fixed dimension, this makes it possible to determine the sign of $I_{d+1}(t)-I_d(t)$ over wide ranges for $t$ when $d$ is reasonably small. Note that according to Theorem~\ref{DSH.sec.3.thm.1}, this computation is not needed for an interval $U_{d,i}\cap{U_{d,j}}$ such that $\Delta(t)$ is at least $d$ for every number $t$ in $U_{d,i}\cap{U_{d,j}}$. This can also be checked easily using symbolic computation. Indeed, $\Delta$ is a convex function of $t$ on each of the intervals $[0,\gamma^-[$, $]\gamma^-,\gamma^+[$, and $]\gamma^+,+\infty[$. Hence, if $d$ is at least $\Delta(t)$ at both of the extremities of $U_{d,i}\cap{U_{d,j}}$, then $d$ must be at least $\Delta(t)$ for every number $t$ contained in $U_{d,i}\cap{U_{d,j}}$. The proof that $\Delta$ is a convex function of $t$ on $[0,\gamma^-[$, $]\gamma^-,\gamma^+[$, and $]\gamma^+,+\infty[$ can be done via elementary calculus techniques and is omitted.

The sign of $I_{d+1}(t)-I_d(t)$ has been determined using this strategy for all dimensions up to $450$ when $0\leq{t}\leq1/\sqrt{2}$, which allowed to extend the monotonicity property stated by Theorem \ref{DSH.sec.3.thm.1} as follows.

\begin{thm}\label{DSH.sec.3.thm.2}
If $d$ is an integer at least $5$, then $I_{d+1}(t)-I_d(t)$ is positive when $0\leq{t}\leq0.20916$ and negative when $\delta<t\leq0.64607$ where $\delta$ is the root of
$$
575\sqrt{5}-528\sqrt{6} -120\bigl(25\sqrt{5}-24\sqrt{6}\bigr)t^2+240\bigl(25\sqrt{5}-36\sqrt{6}\bigr)t^4+17280t^5
$$
that satisfies $0.222924<\delta<0.222925$.
\end{thm}

Note that while both of the ranges for $t$ identified in the statement of Theorem \ref{DSH.sec.3.thm.2} can be extended to larger values by performing the computation in dimensions above $450$, their lower bounds are sharp. In particular, when $t$ is less than but close enough to $\delta$, $I_6(t)-I_5(t)$ is positive.

For the first four positive dimensions, the situation is somewhat more complicated than in dimension five or more. Indeed, when $d$ is less that $5$, the sign of $I_{d+1}(t)-I_d(t)$ changes several times over the intervals $[0,0.20916]$ and $]\delta,0.64607]$. However, it is still possible to perform a case-by-case analysis. Several theorems are now given, that follow from Theorem \ref{DSH.sec.3.thm.2} by analyzing the sign of $I_{d+1}(t)-I_d(t)$ when $1\leq{d}\leq4$. Using (\ref{DSH.sec.3.thm.2.eq.1}), this analysis is straightforward and therefore omitted. In the statements of these theorems, $\alpha_{i,j}^-$, $\alpha_{i,j}^\circ$, and $\alpha_{i,j}^+$ denote certain of the values of $t$ such that $I_i(t)$ and $I_j(t)$ coincide, with the convention that $\alpha_{i,j}^-<\alpha_{i,j}^\circ<\alpha_{i,j}^+$. Each of these values will be described as a root of of an explicit polynomial. Theorem \ref{DSH.sec.3.thm.3}, which is stated in the introduction, is the first consequence of Theorem \ref{DSH.sec.3.thm.2}. It provides two intervals over which $I_d(t)$ is a strictly monotonic function of $d$ over the positive integers. Note that in the statement of Theorem \ref{DSH.sec.3.thm.3}, $\alpha_{3,4}^\circ$ is the third smallest non-negative number $t$ such that $I_3(t)$ is equal to $I_4(t)$ as there is a unique number $t$ in $]\alpha_{3,4}^-,\alpha_{3,4}^\circ[$ such that $I_4(t)-I_3(t)$ vanishes, but this will not play a role here. 

\begin{rem}
It is likely that $I_d(t)$ is also a strictly increasing function of $d$ over the positive integers when $\alpha_{3,4}^+<t<1/\sqrt{2}$ where $\alpha_{3,4}^+$ is the solution of
$$
64-27\sqrt{3}-84t+12\Bigl(16-3\sqrt{3}\Bigr)t^2-64t^3=0
$$
that satisfies $0.697308<\alpha_{3,4}^+<0.697309$. However, in that range for $t$,
$$
e^{6t^2}\frac{12+10t+7t^2+10t^3}{|1-24t^2+48t^4|}
$$
can get greater than $702$ and checking the monotonicity of $I_d(t)$ up to such values of $d$ becomes intractable. Further note that, if $t$ does not satisfy $\alpha_{2,3}^-<t<\alpha_{3,4}^-$, $\alpha_{3,4}^\circ<t\leq0.5$, or $\alpha_{3,4}^+<t<1/\sqrt{2}$, then $I_d(t)$ cannot be a strictly monotonic function of $d$ over the positive integers. Indeed, $I_2(t)$ is at least $I_3(t)$ when $0\leq{t}\leq\alpha_{2,3}^-$, $I_3(t)$ is at least $I_4(t)$ when $\alpha_{3,4}^-\leq{t}\leq\gamma^-$, and $I_4(t)$ is at least $\min\{I_2(t),I_3(t)\}$ when $\gamma^-\leq{t}\leq\alpha_{3,4}^\circ$. Moreover,
$$
I_1(t)<I_4(t)\leq\max\{I_2(t),I_3(t)\}
$$
when $0.5\leq{t}\leq\alpha_{3,4}^+$ and $I_1(t)$ coincides with $I_2(t)$ when $t\geq1/\sqrt{2}$.
\end{rem}

Consider a hyperplane of $\mathbb{R}^d$ whose distance to the center of $[0,1]^d$ is equal to $t$. A consequence of Theorem \ref{DSH.sec.3.thm.3}, is that, if $\alpha_{2,3}^-<t<\alpha_{3,4}^-$, then the volume of $H\cap[0,1]^d$ cannot be maximal when $H$ is orthogonal to the diagonal of a proper face of the hypercube $[0,1]^d$. This volume could still be maximal when $H$ is orthogonal to a diagonal of $[0,1]^d$ itself, though. Likewise, Theorem \ref{DSH.sec.3.thm.3} implies that when $\alpha_{3,4}^\circ<t\leq0.5$, the volume of $H\cap[0,1]^d$ cannot be minimal when $H$ is orthogonal to the diagonal of a proper face of $[0,1]^d$.

The following theorem provides the supremum of $I_d(t)$ when $d$ ranges over the positive integers and $t$ belongs to $[0,0.20916]\,\cup\,]\delta,0.64607]$. Just like Theorem~\ref{DSH.sec.3.thm.3}, it is derived from Theorem~\ref{DSH.sec.3.thm.2} by studying the sign of $I_{d+1}(t)-I_d(t)$ when $1\leq{d}\leq4$. This straightforward study is omitted here.

\begin{thm}\label{DSH.sec.3.thm.4}
The four following statements hold.
\begin{enumerate}
\item[(i)] If $0\leq{t}\leq\alpha_{2,\infty}^-$ or $0.5<t\leq\alpha_{2,3}^+$, then the maximum of $I_d(t)$ when $d$ ranges in $\mathbb{N}\mathord{\setminus}\{0\}$ is equal to $I_2(t)$, where $\alpha_{2,\infty}^-$ is the solution of
$$
\sqrt{2}-2t=\sqrt{\frac{6}{\pi}}e^{-6t^2}
$$
that satisfies $0.0173679<\alpha_{2,\infty}^-<0.017368$ while $\alpha_{2,3}^+$ is equal to
$$
\frac{5+2\sqrt{6\sqrt{6}-14}}{6\sqrt{3}}
$$
and satisfies $0.641788<\alpha_{2,3}^+<0.641789$.
\item[(ii)] If $\alpha_{3,\infty}^-\leq{t}\leq0.20916$ or $\delta\leq{t}\leq\alpha_{1,3}$, or $\alpha_{2,3}^+\leq{t}\leq0.64607$ then the maximum of $I_d(t)$ when $d$ ranges over the positive integers is equal to $I_3(t)$ where $\alpha_{3,\infty}^-$ is the solution of the equation
$$
\frac{3\sqrt{3}}{4}-3\sqrt{3}t^2=\sqrt{\frac{6}{\pi}}e^{-6t^2}
$$
that satisfies $0.192472<\alpha_{3,\infty}^-<0.192473$ while $\alpha_{1,3}$ is equal to
$$
\frac{\sqrt{9-4\sqrt{3}}}{6}
$$
and satisfies $0.239895<\alpha_{1,3}<0.239896$.
\item[(iii)] If $\alpha_{1,3}\leq{t}\leq0.5$, then the maximum of $I_d(t)$ when $d$ ranges over the positive integers is equal to $I_1(t)$.
\item[(iv)] If $\alpha_{2,\infty}^-\leq{t}\leq\alpha_{3,\infty}^-$, then the supremum of $I_d(t)$ when $d$ ranges over the positive integers is equal to $\sqrt{6/\pi}e^{-6t^2}$.
\end{enumerate}
\end{thm}

Recall that $\Delta(t)$ goes to infinity as $t$ goes to $\gamma^-$ and determining the variations of $I_{d+1}(t)-I_d(t)$ for all $d$ less than $\Delta(t)$ becomes intractable. Hence, extending Theorem \ref{DSH.sec.3.thm.4} when $t$ gets close to $\gamma^-$ is not possible with the same techniques. However, it is likely that, when $t$ belongs to $]0.20916,\delta]$, the supremum of $I_d(t)$ when $d$ ranges in $\mathbb{N}\mathord{\setminus}\{0\}$ is still equal to $I_3(t)$. Theorem \ref{DSH.sec.3.cor.2}, stated in the introduction, is now established as a consequence of Theorem \ref{DSH.sec.3.thm.4}.

\begin{proof}[Proof of Theorem \ref{DSH.sec.3.cor.2}]
Assume that $0\leq{t}<0.5$. In this case, the left-hand side of (\ref{DSH.sec.3.cor.2.eq.1}) is a lower bound on the maximal volume of $H\cap[0,1]^d$. In fact, it has been shown by Hermann K{\"o}nig and Alexander Koldobsky that this lower bound is sharp when $d$ is equal to $3$ \cite{KonigKoldobsky2011}. Therefore, according to Theorem \ref{DSH.sec.3.thm.4}, it suffices to observe that the left-hand side of (\ref{DSH.sec.3.cor.2.eq.1}) is greater than $I_1(t)$, $I_2(t)$, $I_3(t)$, and than the right-hand side of (\ref{DSH.sec.3.cor.2.eq.1}) when $t$ does not belong to $[\beta^-,\beta^+]$. The first three inequalities follow from (\ref{DSH.sec.3.thm.2.eq.1}) by a straightforward case-by-case analysis. The fourth inequality is obtained by determining the sign of
\begin{equation}\label{DSH.sec.3.cor.2.eq.2}
\log\biggl(\frac{\sqrt{2-4t^2}-2t}{1-4t^2}\biggr)-\log\Biggl(\sqrt{\frac{6}{\pi}}\Biggr)+6t^2\mbox{.}
\end{equation}

This can be done by observing that this quantity is a continuously differentiable function of $t$ on $[0,0.5[$ and that its derivative has a unique non-negative root with an explicit expression. Moreover, (\ref{DSH.sec.3.cor.2.eq.2}) is negative when $t$ is equal to this root and positive when $t$ is equal to $0$ or to $0.5$. Hence, there exist two number $\beta^-$ and $\beta^+$ such that (\ref{DSH.sec.3.cor.2.eq.2}) is positive when $0\leq{t}<\beta^-$ or $\beta^-<t<0.5$. Estimating the values of $\beta^-$ and $\beta^+$ shows that $0.0181611<\beta^-<0.0181612$ and $0.165625<\beta^+<0.165626$, which completes the proof.
\end{proof}

The infimum of $I_d(t)$  is now given when $d$ ranges over the positive integers and $t$ belongs to $[0,0.20916]$ or $]\delta,0.5]$. Just like Theorem~\ref{DSH.sec.3.thm.4}, this is a consequence of Theorem~\ref{DSH.sec.3.thm.2} and of a study of the sign of $I_{d+1}(t)-I_d(t)$ when $1\leq{d}\leq4$ which is omitted. Note that, when $t$ is greater than $0.5$, this infinmum is equal to $0$. Hence, this infimum is known for every non-negative number $t$ except when $t$ belongs to the interval $]0.20916,\delta]$ around $\gamma^-$.

\begin{thm}\label{DSH.sec.3.thm.5}
The three following statements hold.
\begin{enumerate}
\item[(i)] If $t$ satisfies $0\leq{t}\leq1/\sqrt{2}-1/2$ then the minimum of $I_d(t)$ when $d$ ranges over the positive integers is equal to $I_1(t)$.
\item[(ii)] If $1/\sqrt{2}-1/2\leq{t}\leq0.20916$ or $\delta<t\leq\alpha_{2,\infty}^\circ$, then the minimum of $I_d(t)$ when $d$ ranges over the positive integers is equal to $I_2(t)$ where $\alpha_{2,\infty}^\circ$ denotes the solution of the equation
$$
\sqrt{2}-2t=\sqrt{\frac{6}{\pi}}e^{-6t^2}
$$
such that $0.290166<\alpha_{2,\infty}^\circ<0.290167$.
\item[(iii)] If $\alpha_{2,\infty}^\circ\leq{t}\leq0.5$, then the infimum of $I_d(t)$ when $d$ ranges over the positive integers is equal to $\sqrt{6/\pi}e^{-6t^2}$.
\end{enumerate}
\end{thm}

\section{The limit of a more general family of integrals}\label{DSH.sec.4}

In this section, $k$ is a fixed non-negative even integer and the integral
\begin{equation}\label{DSH.sec.4.eq.0}
I_{d,k}(t)=\frac{2\sqrt{d}}{\pi}\!\int_{0}^{+\infty}\biggl(\frac{\sin\,u}{u}\biggr)^{\!d}\!\cos\Bigl(2\sqrt{d}tu\Bigr)u^kdu\mbox{.}
\end{equation}
that generalizes $I_d(t)$ is considered. The following theorem is established.

\begin{thm}\label{DSH.sec.4.thm.1}
Consider a non-negative, even integer $k$. If $d$ is an integer at least $k+2$, then for every non-negative number $t$,
\begin{multline*}
\biggl|I_{d,k}(t)-\frac{T_k(t)}{\sqrt{d}^k}\biggr|\leq\sqrt{\frac{6}{\pi}}\biggl(\frac{26\sqrt{6}}{25\sqrt{d}}\biggr)^{\!k+2}\Biggl(1+\sqrt{\frac{6}{\pi}}t\Biggr)(k/2+1)!\\
\hfill+\frac{12}{\pi\sqrt{d}e^{d/6}}+2\sqrt{d}\biggl(\frac{\pi}{2}\biggr)^{\!k}\biggl(\frac{17}{20}\biggr)^{\!d}\mbox{.}
\end{multline*}
\end{thm}

The proof of Theorem \ref{DSH.sec.4.thm.1} relies on the same general strategy than that of Theorem~\ref{DSH.sec.1.thm.0}. Recall that $\psi$ is the inverse of the function $\phi$ given by (\ref{DSH.sec.1.eq.0.1}). As discussed in Section \ref{DSH.sec.1}, $\psi$ is a strictly increasing analytic function on $[0,1]$ such that $\psi(0)$ vanishes and $\psi(1)$ is positive and less than $\pi/2$. Moreover, (\ref{DSH.sec.1.eq.6}) provides the first terms of its Taylor series expansion around $0$. Denote
\begin{equation}\label{DSH.sec.4.eq.0.5}
F_{d,k}(t)=\frac{2\sqrt{d}}{\pi}\!\int_0^{\psi(1)}\biggl(\frac{\sin\,u}{u}\biggr)^{\!d}\!\cos\Bigl(2\sqrt{d}tu\Bigr)u^kdu\mbox{.}
\end{equation}

The change of variables $u=\psi(x)$ and the trigonometric identity (\ref{DSH.sec.1.eq.7.25}) yield
\begin{multline}\label{DSH.sec.4.eq.1}
F_{d,k}(t)=\frac{2\sqrt{d}}{\pi}\!\int_0^1e^{-dx^2\!/6}\!\cos\Bigl(2\sqrt{d}tx\Bigr)c_d(x,t)\psi^k(x)dx\\
\hfill-\frac{2\sqrt{d}}{\pi}\!\int_0^1e^{-dx^2\!/6}\!\sin\Bigl(2\sqrt{d}tx\Bigr)s_d(x,t)\psi^k(x)dx
\end{multline}
where $c_d(x,t)$ and $s_d(x,t)$ are given by (\ref{DSH.sec.1.eq.7.6}). By the following lemma, $F_{d,k}(t)$ gets exponentially close to $I_{d,k}(t)$ when $d$ goes to infinity.
\begin{lem}\label{DSH.sec.4.lem.0}
For every integer $d$ at least $k+2$ and non-negative number $t$,
$$
\bigl|I_{d,k}(t)-F_{d,k}(t)\bigr|\leq2\sqrt{d}\biggl(\frac{\pi}{2}\biggr)^{\!k}\biggl(\frac{17}{20}\biggr)^{\!d}\mbox{.}
$$
\end{lem}
\begin{proof}
This follows immediately from Lemma \ref{DSH.sec.1.lem.0} and from (\ref{DSH.sec.1.eq.4.5}).
\end{proof}

The integrands in the right-hand side of (\ref{DSH.sec.4.eq.1}) are now estimated.

\begin{lem}\label{DSH.sec.4.lem.1}
For every positive integer $d$, every non-negative number $t$, and every number $x$ contained in the interval $[0,1]$,
$$
\Big|c_d(x,t)\psi^k(x)-x^k\Bigr|\leq\biggl(\frac{26x}{25}\biggr)^{\!k+2}+\sqrt{d}tx^{k+3}\mbox{.}
$$
\end{lem}
\begin{proof}
Recall that $\psi$ is analytic on the interval $[0,1]$ and that the first terms of its Taylor series expansion around $0$ are given by (\ref{DSH.sec.1.eq.6}). Hence, one obtains by Taylor series approximations of $\psi$ and its derivative with Lagrange form remainders that, for every number $x$ contained in $[0,1]$,
\begin{equation}\label{DSH.sec.4.lem.1.eq.1}
\psi(x)=x+\frac{d^3\psi}{dx^3}(\eta)\frac{x^3}{6}
\end{equation}
where $0\leq\eta\leq{x}$ and
\begin{equation}\label{DSH.sec.4.lem.1.eq.2}
\frac{d\psi}{dx}(x)=1+\frac{d^3\psi}{dx^3}(\xi)\frac{x^2}{2}
\end{equation}
where $0\leq\xi\leq{x}$. Similarly, a first order Taylor series approximation of $\cos z$ around $0$ with a Lagrange form remainder,
$$
\cos z=1-z\sin\zeta
$$
where $\zeta$ is a real number that depends on $z$. Substituting
$$
z=2\sqrt{d}t\bigl(\psi(x)-x\bigr)\mbox{,}
$$
within this equality yields
\begin{equation}\label{DSH.sec.4.lem.1.eq.3}
\cos\Bigl(2\sqrt{d}t\bigl(\psi(x)-x\bigr)\Bigr)=1-2\sqrt{d}t\bigl(\psi(x)-x\bigr)\sin\zeta
\end{equation}
where $\zeta$ depends on $d$, $t$, and $x$. One obtains by combining the equalities (\ref{DSH.sec.4.lem.1.eq.1}), (\ref{DSH.sec.4.lem.1.eq.2}), and (\ref{DSH.sec.4.lem.1.eq.3}) that, for every $x$ contained in $[0,1]$,
$$
c_d(x,t)=1+\frac{d^3\psi}{dx^3}(\xi)\frac{x^2}{2}-\sqrt{d}t\frac{d^3\psi}{dx^3}(\eta)\frac{x^3}{3}\frac{d\psi}{dx}(x)\sin\zeta\mbox{.}
$$

Now consider a positive integer $k$ and observe that, by (\ref{DSH.sec.4.lem.1.eq.1}),
$$
\psi^k(x)=x^k+x^{k+2}\sum_{i=1}^k{k\choose{i}}\biggl(\frac{d^3\psi}{dx^3}(\eta)\biggr)^{\!i}\frac{x^{2i-2}}{6^i}\mbox{.}
$$

As a consequence, for every number $x$ in $[0,1]$,
\begin{multline*}
c_d(x,t)\psi^k(x)=x^k+\frac{d^3\psi}{dx^3}(\xi)\frac{x^{k+2}}{2}-\sqrt{d}t\frac{d^3\psi}{dx^3}(\eta)\frac{x^{k+3}}{3}\frac{d\psi}{dx}(x)\sin\zeta\\
\hfill+\cos\Bigl(2\sqrt{d}t\bigl(\psi(x)-x\bigr)\Bigr)\frac{d\psi}{dx}(x)x^{k+2}\sum_{i=1}^k{k\choose{i}}\biggl(\frac{d^3\psi}{dx^3}(\eta)\biggr)^{\!i}\frac{x^{2i-2}}{6^i}
\end{multline*}
and in turn, since $x$ belongs to $[0,1]$,
\begin{multline*}
\Bigl|c_d(x,t)\psi^k(x)-x^k\Bigr|\leq\biggl|\frac{d^3\psi}{dx^3}(\xi)\biggr|\frac{x^{k+2}}{2}+\biggl|\frac{d^3\psi}{dx^3}(\eta)\frac{d\psi}{dx}(x)\biggr|\frac{\sqrt{d}tx^{k+3}}{3}\\
\hfill+\biggl|\frac{d\psi}{dx}(x)\biggr|x^{k+2}\sum_{i=1}^k{k\choose{i}}\biggl|\frac{1}{6}\frac{d^3\psi}{dx^3}(\eta)\biggr|^i\mbox{.}
\end{multline*}

Now observe that
$$
\sum_{i=1}^k{k\choose{i}}\biggl|\frac{1}{6}\frac{d^3\psi}{dx^3}(\eta)\biggr|^i<\biggl(\frac{1}{6}\biggl|\frac{d^3\psi}{dx^3}(\eta)\biggr|+1\biggr)^{\!k}\mbox{.}
$$

Hence, for every number $x$ in $[0,1]$,
\begin{multline*}
\Bigl|c_d(x,t)\psi^k(x)-x^k\Bigr|\leq\Biggl(\frac{1}{2}\biggl|\frac{d^3\psi}{dx^3}(\xi)\biggr|+\biggl|\frac{d\psi}{dx}(x)\biggr|\biggl(\frac{1}{6}\biggl|\frac{d^3\psi}{dx^3}(\eta)\biggr|+1\biggr)^{\!k}\Biggr)x^{k+2}\\
\hfill+\biggl|\frac{d^3\psi}{dx^3}(\eta)\frac{d\psi}{dx}(x)\biggr|\frac{\sqrt{d}tx^{k+3}}{3}\mbox{.}
\end{multline*}

By Proposition \ref{DSH.sec.3.prop.1}, the coefficients of $x^{k+2}$ and $\sqrt{d}tx^{k+3}$ in the right-hand side are at most $(26/25)^{k+2}$ and $1$, respectively and the lemma follows.
\end{proof}

The second integrand in the right-hand side of (\ref{DSH.sec.4.eq.1}) is estimated as follows.

\begin{lem}\label{DSH.sec.4.lem.2}
For every positive integer $d$, every non-negative number $t$, and every number $x$ contained in the interval $[0,1]$,
$$
\Big|s_d(x,t)\psi^k(x)\Bigr|\leq\Biggl(\biggl(\frac{26}{25}\biggr)^{\!k+2}-1\Biggr)\sqrt{d}tx^{k+3}\mbox{.}
$$
\end{lem}
\begin{proof}
It follows from order $1$ and $3$ Taylor series approximations of $\psi$ with Lagrange form remainders that, for every $x$ in $[0,1]$,
\begin{equation}\label{DSH.sec.4.lem.2.eq.1}
\psi(x)=\frac{d\psi}{dx}(\eta_1)x
\end{equation}
where $0\leq\eta_1\leq{x}$ and
\begin{equation}\label{DSH.sec.4.lem.2.eq.2}
\psi(x)=x+\frac{d^3\psi}{dx^3}(\eta_3)\frac{x^3}{6}
\end{equation}
where $0\leq\eta_3\leq{x}$. Likewise, one obtains from a first order Taylor series approximation of $\sin z$ around $0$ with a Lagrange form remainder that
$$
\sin z=z\cos\zeta
$$
where $\zeta$ is a real number that depends on $z$. By the change of variables
$$
z=2\sqrt{d}t\bigl(\psi(x)-x\bigr)\mbox{,}
$$
this can be rewritten into
\begin{equation}\label{DSH.sec.4.lem.2.eq.3}
\sin\Bigl(2\sqrt{d}t\bigl(\psi(x)-x\bigr)\Bigr)=2\sqrt{d}t\bigl(\psi(x)-x\bigr)\cos\zeta
\end{equation}
where $\zeta$ depends on $d$, $t$, and $x$. Now consider a positive integer $k$. Combining (\ref{DSH.sec.4.lem.2.eq.1}),  (\ref{DSH.sec.4.lem.2.eq.2}), and (\ref{DSH.sec.4.lem.2.eq.3}) implies that, for every $x$ contained in $[0,1]$,
$$
s_d(x,t)\psi^k(x)=\sqrt{d}t\frac{d^3\psi}{dx^3}(\eta_3)\biggl(\frac{d\psi}{dx}(\eta_1)\biggr)^{\!k}\frac{x^{k+3}}{3}\cos\zeta\mbox{.}
$$

Therefore, for every number $x$ contained in $[0,1]$,
$$
\Bigl|s_d(x,t)\psi^k(x)\Bigr|\leq\biggl|\frac{d^3\psi}{dx^3}(\eta_3)\biggr|\biggl|\frac{d\psi}{dx}(\eta_1)\biggr|^{k}\frac{\sqrt{d}tx^{k+3}}{3}\mbox{.}
$$

Bounding the coefficient of $\sqrt{d}tx^{k+3}$ in the right-hand side of this inequality by means of Proposition \ref{DSH.sec.3.prop.1} completes the proof.
\end{proof}

Theorem \ref{DSH.sec.4.thm.1} is now established.

\begin{proof}[Proof of Theorem \ref{DSH.sec.4.thm.1}]
Let $d$ and $k$ be two integers such that $d$ is positive and $k$ non-negative. For any non-negative number $t$, denote
\begin{equation}\label{DSH.sec.4.thm.1.eq.2}
J_{d,k}(t)=\frac{2\sqrt{d}}{\pi}\!\int_0^1e^{-dx^2\!/6}\cos\Bigl(2\sqrt{d}tx\Bigr)x^kdx\mbox{.}
\end{equation}

According to (\ref{DSH.sec.4.eq.1}) and the triangle inequality,
\begin{multline*}
\bigl|F_{d,k}(t)-J_{d,k}(t)\bigr|\leq\frac{2\sqrt{d}}{\pi}\!\int_0^1e^{-dx^2\!/6}\Bigl|c_d(x,t)\psi^k(x)-x^k\Bigr|dx\\
\hfill+\frac{2\sqrt{d}}{\pi}\!\int_0^1e^{-dx^2\!/6}\Bigl|s_d(x,t)\psi^k(x)\Bigr|dx\mbox{.}
\end{multline*}

Lemmas \ref{DSH.sec.4.lem.1} and \ref{DSH.sec.4.lem.2} then imply that
\begin{multline*}
\bigl|F_{d,k}(t)-J_{d,k}(t)\bigr|\leq\frac{2\sqrt{d}}{\pi}\biggl(\frac{26}{25}\biggr)^{\!k+2}\Biggl(\int_0^1e^{-dx^2\!/6}x^{k+2}dx\\
\hfill+\sqrt{d}t\!\int_0^1e^{-dx^2\!/6}x^{k+3}dx\Biggr)\mbox{.}
\end{multline*}
and if $k$ is even, it follows from (\ref{DSH.sec.1.thm.0.eq.1.375}) and (\ref{DSH.sec.1.thm.0.eq.1.4}) that
$$
\bigl|F_{d,k}(t)-J_{d,k}(t)\bigr|\leq\sqrt{\frac{6}{\pi}}\biggl(\frac{26\sqrt{3}}{25\sqrt{2d}}\biggr)^{\!k+2}\Biggl(\frac{(k+2)!}{(k/2+1)!}+2^{k+2}(k/2+1)!\sqrt{\frac{6}{\pi}}t\Biggr)\mbox{.}\\
$$

However, observe that
$$
\frac{(k+2)!}{(k/2+1)!}\leq2^{k+2}(k/2+1)!
$$
and as a consequence
\begin{equation}\label{DSH.sec.4.thm.1.eq.4}
\bigl|F_{d,k}(t)-J_{d,k}(t)\bigr|\leq\sqrt{\frac{6}{\pi}}\biggl(\frac{26\sqrt{6}}{25\sqrt{d}}\biggr)^{\!k+2}\Biggl(1+\sqrt{\frac{6}{\pi}}t\Biggr)(k/2+1)!\\
\end{equation}

One obtains from the change of variables $y=\sqrt{d}x$ in (\ref{DSH.sec.4.thm.1.eq.2}) that
$$
J_{d,k}(t)=\frac{2}{\pi\sqrt{d}^k}\!\int_0^{\sqrt{d}}e^{-y^2\!/6}\cos(2ty)y^kdy\mbox{.}
$$

Therefore by (\ref{DSH.sec.1.eq.6.7}) and the triangle inequality,
$$
\biggl|J_{d,k}(t)-\frac{T_k(t)}{\sqrt{d}^k}\biggr|\leq\frac{2}{\pi\sqrt{d}^k}\!\int_{\sqrt{d}}^{+\infty}e^{-y^2\!/6}y^kdy\mbox{.}
$$

In turn, according to Proposition \ref{DSH.sec.1.prop.2},
\begin{equation}\label{DSH.sec.4.thm.1.eq.5}
\biggl|J_{d,k}(t)-\frac{T_k(t)}{\sqrt{d}^k}\biggr|\leq\frac{12}{\pi\sqrt{d}e^{d/6}}
\end{equation}
when $d$ is at least $\sqrt{6k}$. One obtains by combining~(\ref{DSH.sec.4.thm.1.eq.4}), (\ref{DSH.sec.4.thm.1.eq.5}), and the bound provided by Lemma~\ref{DSH.sec.4.lem.0} with the triangle inequality that if $d$ is at least $k+2$ then for every non-negative number $t$, the desired inequality holds. This follows in particular from the observation that $k+2$ is greater than $\sqrt{6k}$.
\end{proof}

Theorem \ref{DSH.sec.4.thm.1} will be used in the next section in the special case when $k$ is equal to $4$. The following statement treats this case.

\begin{cor}\label{DSH.sec.4.cor.1}
If $d$ is at least $124$ then, for every non-negative number $t$,
$$
\Biggl|I_{d,4}(t)-27\sqrt{\frac{6}{\pi}}\frac{1-24t^2+48t^4}{d^2e^{6t^2}}\Biggr|<\frac{2267+3132t}{d^3}\mbox{.}
$$
\end{cor}
\begin{proof}
It follows from Theorem \ref{DSH.sec.4.thm.1}, from (\ref{DSH.sec.1.eq.7}), and from (\ref{DSH.sec.1.eq.8}) that, if $d$ is at least $6$, then, for every non-negative number $t$,
\begin{multline*}
\Biggl|I_{d,4}(t)-27\sqrt{\frac{6}{\pi}}\frac{1-24t^2+48t^4}{d^2e^{6t^2}}\Biggr|\leq\sqrt{\frac{6}{\pi}}\Biggl(1+\sqrt{\frac{6}{\pi}}t\Biggr)\biggl(\frac{26}{25}\biggr)^{\!6}\frac{1296}{d^3}\\
\hfill+\frac{12}{\pi\sqrt{d}e^{d/6}}+\frac{\pi^4\sqrt{d}}{8}\biggl(\frac{17}{20}\biggr)^{\!d}\mbox{.}
\end{multline*}

However, when $d$ is at least $124$,
$$
\frac{12}{\pi\sqrt{d}e^{d/6}}+\frac{\pi^4\sqrt{d}}{8}\biggl(\frac{17}{20}\biggr)^{\!d}<\frac{1}{2d^3}\mbox{.}
$$
and as a consequence, for these values of $d$,
\begin{multline*}
\Biggl|I_{d,4}(t)-27\sqrt{\frac{6}{\pi}}\frac{1-24t^2+48t^4}{d^2e^{6t^2}}\Biggr|\leq\Biggl(1296\sqrt{\frac{6}{\pi}}\biggl(\frac{26}{25}\biggr)^{\!6}+\frac{1}{2}\Biggr)\frac{1}{d^3}\\
\hfill+\frac{7776}{\pi}\biggl(\frac{26}{25}\biggr)^{\!6}\frac{t}{d^3}\mbox{.}
\end{multline*}

The desired inequality is then obtained because the coefficients of $1/d^3$ and $t/d^3$ in the right-hand side are less than $2267$ and $3132$, respectively.
\end{proof}

\section{The local extremality of diagonal sections}\label{DSH.sec.5}

Throughout this section, $t$ is a fixed non-negative number and $H$ is a hyperplane of $\mathbb{R}^d$ whose distance to the center of $[0,1]^d$ is equal to $t$. Thanks to the symmetries of the hypercube, it can be assumed that
$$
H=\biggl\{x\in\mathbb{R}^d:a\mathord{\cdot}x=\frac{\sigma(a)}{2}-t\biggr\}
$$
where $a$ is a vector in $\mathbb{S}^{d-1}\cap[0,+\infty[^d$ whose sum of coordinates is denoted by $\sigma(a)$. The $(d-1)$-dimensional volume of $H\cap[0,1]^d$ is denoted by $V$ and considered a function of $a$. Note that $H$ is orthogonal to a diagonal of a $n$\nobreakdash-dimensional face $F$ of $[0,1]^d$ precisely when $a$ is a vector with $d-n$ coordinates equal to $0$ and $n$ coordinates equal to $1/\sqrt{n}$. A criterion for the local extremality of $V$ at any such vector $a$ has been given in~\cite{Pournin2024,Pournin2025}. Throughout the section, $f$ and $g$ denote the functions defined for every non-negative number $u$ by
\begin{equation}\label{DSH.sec.5.eq.1}
\left\{\begin{array}{l}
\displaystyle f(u)=2\frac{\sin^2 u}{u^2}-\cos u\frac{\sin u}{u}-1\mbox{ and}\\[\bigskipamount]
\displaystyle g(u)=\cos u +\biggl(\frac{u^2}{3}-1\biggr)\frac{\sin u}{u}\mbox{.}
\end{array}\right.
\end{equation}

The criterion from~\cite{Pournin2024,Pournin2025} (see Theorem~1.1 from \cite{Pournin2025}, Equation (18) from~\cite{Pournin2024}, and Equation~(4) from \cite{Pournin2025}) can be rephrased as follows.

\begin{thm}\label{DSH.sec.5.thm.1}
Consider an integer $n$ satisfying $4\leq{n}\leq{d}$. Denote
\begin{equation}\label{DSH.sec.5.thm.1.eq.1}
r=\int_{0}^{+\infty}\biggl(\frac{\sin\,u}{u}\biggr)^{\!n-2}\!\cos\Bigl(2\sqrt{d}tu\Bigr)f(u)du
\end{equation}
and
\begin{equation}\label{DSH.sec.5.thm.1.eq.0.5}
s=-\int_{0}^{+\infty}\biggl(\frac{\sin\,u}{u}\biggr)^{\!n-1}\!\cos\Bigl(2\sqrt{d}tu\Bigr)g(u)du\mbox{.}
\end{equation}

Assume that $a$ is a vector with $d-n$ coordinates equal to $0$ and whose other $n$ coordinates are equal to $1/\sqrt{n}$. At that vector, $V$ is
\begin{enumerate}
\item[(i)] strictly locally maximal if $n$ coincides with $d$ and $r$ is negative or $n$ is less than $d$ and $r$ and $s$ are both negative,
\item[(ii)] strictly locally minimal if $n$ coincides with $d$ and $r$ is positive or $n$ is less than $d$ and $r$ and $s$ are both positive, or
\item[(iii)] not locally extremal if $n$ is less than $d$ and $r$ and $s$ are opposite.
\end{enumerate}
\end{thm}

Using this criterion, it is shown in~\cite{Pournin2024} that, if $d$ is at least $4$ and $t$ is small  enough, then $V$ is always strictly locally maximal when $H$ is orthogonal to a diagonal of $[0,1]^d$. In contrast, $V$ is never locally extremal when $H$ is orthogonal to a diagonal of a \emph{proper} face of $[0,1]^d$ of dimension at least $3$ and $t$ is equal to $0$~\cite{AmbrusGargyan2024b}. In this section, these results are extended to a range of values of $t$ away from $0$ and to all fixed $t$ when the dimension goes to infinity. The strategy is similar to the one used in Section \ref{DSH.sec.3}: it consists in combining Theorem~\ref{DSH.sec.5.thm.1} with the results of Section \ref{DSH.sec.4} for dimensions above a certain explicit threshold and by using symbolic computation for dimensions below that threshold. 

Note that according to (\ref{DSH.sec.5.eq.1}), $f$ and $g$ are analytic functions of $u$ on $\mathbb{R}$ and, by order $6$ Taylor series approximations around $0$,
\begin{equation}\label{DSH.sec.5.eq.2}
\left\{\begin{array}{l}
\displaystyle f(u)=-\frac{2u^4}{45}+O(u^6)\mbox{ and}\\[\bigskipamount]
\displaystyle g(u)=-\frac{u^4}{45}+O(u^6)\mbox{.}
\end{array}\right.
\end{equation}

Bounding the error term in these equalities provides the following.

\begin{prop}\label{DSH.sec.5.prop.1}
For every number $u$ satisfying $0\leq{u}\leq1$,
$$
\Biggl|f(u)+\frac{2u^4}{45}\Biggr|\leq\frac{2u^6}{315}
$$
and
$$
\Biggl|g(u)+\frac{u^4}{45}\Biggr|\leq\frac{u^6}{630}\mbox{.}
$$

\end{prop}
\begin{proof}
Since $f$ is analytic on $\mathbb{R}$ and given (\ref{DSH.sec.5.eq.2}), the function $\tilde{f}$ such that $\tilde{f}(u)$ is equal to $-2/45$ when $u$ is equal to $0$ and to $f(u)/u^4$ otherwise is also an analytic function of $u$ on $\mathbb{R}$. Likewise, the function $\tilde{g}$ such that $\tilde{g}(0)$ is equal to $-1/45$ and $\tilde{g}(u)$ is equal to $g(u)/u^4$ when $u$ is non-zero is analytic on $\mathbb{R}$.

According to~(\ref{DSH.sec.5.eq.2}), order $2$ Taylor series approximations of $\tilde{f}$ and $\tilde{g}$ around $0$ with Lagrange form remainders yield
\begin{equation}\label{DSH.sec.5.prop.1.eq.0}
\left\{\begin{array}{l}
\displaystyle\tilde{f}(u)=-\frac{2}{45}+\frac{d^2\tilde{f}}{du^2}(\xi)\frac{u^2}{2}\mbox{ and}\\[\bigskipamount]
\displaystyle\tilde{g}(u)=-\frac{1}{45}+\frac{d^2\tilde{g}}{du^2}(\eta)\frac{u^2}{2}\mbox{.}
\end{array}\right.
\end{equation}
for every number $u$ in $[0,1]$, where $\xi$ and $\eta$ belong to $[0,u]$.

The second derivatives of $\tilde{f}$ and $\tilde{g}$ are
\begin{equation}\label{DSH.sec.5.prop.1.eq.1}
\frac{d^2\tilde{f}}{du^2}(u)=\frac{42-20u^2-14(3-u^2)\cos2u-(39-2u^2)u\sin2u}{u^8}
\end{equation}
and
\begin{equation}\label{DSH.sec.5.prop.1.eq.2}
\frac{d^2\tilde{g}}{du^2}(u)=\frac{9(10-u^2)u\cos u-(90-39u^2+u^4)\sin u}{3u^7}\mbox{.}
\end{equation}

However, recall that if $u$ is a non-negative number,
\begin{equation}\label{DSH.sec.5.prop.1.eq.3}
1-\frac{u^2}{2}+\frac{u^4}{4!}-\frac{u^6}{6!}+\frac{u^8}{8!}-\frac{u^{10}}{10!}\leq\cos u\leq1-\frac{u^2}{2}+\frac{u^4}{4!}-\frac{u^6}{6!}+\frac{u^8}{8!}
\end{equation}
and
\begin{equation}\label{DSH.sec.5.prop.1.eq.4}
u-\frac{u^3}{6}+\frac{u^5}{5!}-\frac{u^7}{7!}+\frac{u^9}{9!}-\frac{u^{11}}{11!}\leq\sin u\leq u-\frac{u^3}{6}+\frac{u^5}{5!}-\frac{u^7}{7!}+\frac{u^9}{9!}\mbox{.}
\end{equation}

Since $-14(3-u^2)$ and $-(39-2u^2)$ are both negative when $0\leq{u}\leq1$, bounding $\cos 2u$ and $\sin 2u$ in (\ref{DSH.sec.5.prop.1.eq.1}) according to (\ref{DSH.sec.5.prop.1.eq.3}) and (\ref{DSH.sec.5.prop.1.eq.4}) shows that
$$
\frac{8(99u^2-17u^4+2u^6)}{155925}\leq\frac{4}{315}-\frac{d^2\tilde{f}}{du^2}(u)\leq\frac{8(6u^2-u^4)}{2835}
$$
when $0\leq{u}\leq1$. Note that $6u^2-u^4$ and $99u^2-17u^4+2u^6$ are increasing functions of $u$ on the interval $[0,1]$. Hence if $0\leq{u}\leq1$, then
\begin{equation}\label{DSH.sec.5.prop.1.eq.5}
-\frac{4}{2835}\leq\frac{d^2\tilde{f}}{du^2}(u)\leq\frac{4}{315}\mbox{.}
\end{equation}

Likewise, $9(10-u^2)$ is positive and $-(90-39u^2+u^4)$ negative when $0\leq{u}\leq1$. Hence, bounding $\cos u$ and $\sin u$ in (\ref{DSH.sec.5.prop.1.eq.2}) via (\ref{DSH.sec.5.prop.1.eq.3}) and (\ref{DSH.sec.5.prop.1.eq.4}) shows that
$$
\frac{63360u^2-3390u^4+149u^6-u^8}{119750400}\leq\frac{1}{315}-\frac{d^2\tilde{g}}{du^2}(u)\leq\frac{5760u^2-210u^4+u^6}{10886400}\mbox{.}
$$

Again, $5760u^2-210u^4+u^6$ and $63360u^2-3390u^4+149u^6-u^8$ are decreasing functions of $u$ on $[0,1]$. Hence, for every number $u$ satisfying $0\leq{u}\leq1$,
\begin{equation}\label{DSH.sec.5.prop.1.eq.6}
\frac{29009}{10886400}\leq\frac{d^2\tilde{g}}{du^2}(u)\leq\frac{1}{315}\mbox{.}
\end{equation}

Using (\ref{DSH.sec.5.prop.1.eq.5}) and (\ref{DSH.sec.5.prop.1.eq.6}) to bound the second derivatives of $\tilde{f}$ at $\xi$ and of $\tilde{g}$ at $\eta$ in (\ref{DSH.sec.5.prop.1.eq.0}) shows that for every number $u$ satisfying $0\leq{u}\leq1$,
$$
\biggl|\tilde{f}(u)+\frac{2}{45}\biggr|\leq\frac{2u^2}{315}
$$
and
$$
\Biggl|\tilde{g}(u)+\frac{1}{45}\Biggr|\leq\frac{u^2}{630}\mbox{.}
$$

Multiplying these inequalities by $u^4$ completes the proof.
\end{proof}

Proposition \ref{DSH.sec.5.prop.1} states that $f(u)$ and $g(u)$ roughly behave like $u^4$ up to a multiplicative constant when $u$ is close to $0$. This makes it possible to estimate the right-hand sides of (\ref{DSH.sec.5.thm.1.eq.1}) and (\ref{DSH.sec.5.thm.1.eq.0.5}) using Corollary \ref{DSH.sec.4.cor.1}. In fact, Proposition~\ref{DSH.sec.5.prop.1} will not be used directly, but rather the following consequence.

\begin{lem}\label{DSH.sec.5.lem.3}
For every number $u$ satisfying $0\leq{u}\leq1$,
$$
\Biggl|\frac{u^2f(u)}{\sin^2u}+\frac{2u^4}{45}\Biggr|\leq\frac{242u^6}{7875}
$$
and
$$
\Biggl|\frac{ug(u)}{\sin u}+\frac{u^4}{45}\Biggr|\leq\frac{2u^6}{315}\mbox{.}
$$
\end{lem}
\begin{proof}
According to the triangle inequality and to Proposition \ref{DSH.sec.5.prop.1}, it suffices to prove that, for every number $u$ satisfying $0\leq{u}\leq1$,
\begin{equation}\label{DSH.sec.5.lem.3.eq.1}
\Biggl|\biggl(\frac{u^2}{\sin^2u}-1\biggr)f(u)\Biggr|\leq\frac{64u^6}{2625}
\end{equation}
and
\begin{equation}\label{DSH.sec.5.lem.3.eq.1.5}
\Biggl|\biggl(\frac{u}{\sin u}-1\biggr)g(u)\Biggr|\leq\frac{u^6}{210}
\end{equation}

Consider such a number $u$ and recall that
$$
1-\frac{u^2}{6}\leq\frac{\sin u}{u}\leq1\mbox{.}
$$

As a consequence,
$$
\left\{\begin{array}{l}
\displaystyle0\leq\frac{u}{\sin u}-1\leq\frac{u^2}{6-u^2}\mbox{ and}\\[\bigskipamount]
\displaystyle0\leq\frac{u^2}{\sin^2u}-1\leq\frac{12-u^2}{(6-u^2)^2}u^2\mbox{.}
\end{array}\right.
$$

Since $u$ belongs to $[0,1]$, it follows that
\begin{equation}\label{DSH.sec.5.lem.3.eq.2}
\left\{\begin{array}{l}
\displaystyle\Biggl|\biggl(\frac{u^2}{\sin^2u}-1\biggr)f(u)\Biggr|\leq\frac{12u^2}{25}|f(u)|\mbox{ and}\\[1.5\bigskipamount]
\displaystyle\Biggl|\biggl(\frac{u}{\sin u}-1\biggr)g(u)\Biggr|\leq\frac{u^2}{5}|g(u)|\mbox{.}
\end{array}\right.
\end{equation}

However, according to Proposition \ref{DSH.sec.5.prop.1},
\begin{equation}\label{DSH.sec.5.lem.3.eq.3}
|f(u)|\leq\biggl(\frac{2}{45}+\frac{2}{315}\biggr)u^4\mbox{.}
\end{equation}
and
\begin{equation}\label{DSH.sec.5.lem.3.eq.4}
|g(u)|\leq\biggl(\frac{1}{45}+\frac{1}{630}\biggr)u^4\mbox{.}
\end{equation}

Therefore, as
$$
\frac{12}{25}\biggl(\frac{2}{45}+\frac{2}{315}\biggr)=\frac{64}{2625}
$$
and
$$
\frac{1}{5}\biggl(\frac{1}{45}+\frac{1}{630}\biggr)=\frac{1}{210}\mbox{,}
$$
one obtains the desired inequalities from (\ref{DSH.sec.5.lem.3.eq.2}), (\ref{DSH.sec.5.lem.3.eq.3}), and (\ref{DSH.sec.5.lem.3.eq.4}).
\end{proof}

Using this lemma, one can prove the following estimates.

\begin{lem}\label{DSH.sec.5.lem.2}
If $n$ is at least $118$ then, for every non-negative number $t$,
\begin{equation}\label{DSH.sec.5.lem.2.eq.6}
\Biggl|\frac{2\sqrt{n}}{\pi}\!\int_0^{+\infty}\biggl(\frac{\sin\,u}{u}\biggr)^{\!n-2}\!\cos\Bigl(2\sqrt{d}tu\Bigr)f(u)du+\frac{2}{45}I_{n,4}(t)\Biggr|<\frac{21}{n^3}
\end{equation}
and
\begin{equation}\label{DSH.sec.5.lem.2.eq.6.1}
\Biggl|\frac{2\sqrt{n}}{\pi}\!\int_0^{+\infty}\biggl(\frac{\sin\,u}{u}\biggr)^{\!n-1}\!\cos\Bigl(2\sqrt{d}tu\Bigr)g(u)du+\frac{1}{45}I_{n,4}(t)\Biggr|<\frac{6}{n^3}
\end{equation}
\end{lem}
\begin{proof}
Consider an integer $n$ at least $4$ and denote
$$
L=\!\int_0^{\psi(1)}\biggl(\frac{\sin\,u}{u}\biggr)^{\!n}u^6du\mbox{.}
$$

Recall that $0<\psi(1)<1$. Hence, by Lemma \ref{DSH.sec.5.lem.3} and the triangle inequality,
\begin{equation}\label{DSH.sec.5.lem.2.eq.1.5}
\Biggl|\frac{2\sqrt{n}}{\pi}\!\int_0^{\psi(1)}\biggl(\frac{\sin\,u}{u}\biggr)^{\!n-2}\!\cos\Bigl(2\sqrt{d}tu\Bigr)f(u)du+\frac{2}{45}F_{n,4}(t)\Biggr|\leq\frac{484\sqrt{n}}{7875\pi}L
\end{equation}
and
\begin{equation}\label{DSH.sec.5.lem.2.eq.1.6}
\Biggl|\frac{2\sqrt{n}}{\pi}\!\int_0^{\psi(1)}\biggl(\frac{\sin\,u}{u}\biggr)^{\!n-1}\!\cos\Bigl(2\sqrt{d}tu\Bigr)g(u)du+\frac{1}{45}F_{n,4}(t)\Biggr|\leq\frac{4\sqrt{n}}{315\pi}L
\end{equation}
where $F_{n,k}(t)$ is the quantity whose expression is given by (\ref{DSH.sec.4.eq.0.5}). However, one obtains by the change of variables $u=\psi(x)$ that
\begin{equation}\label{DSH.sec.5.lem.2.eq.2}
\frac{\sqrt{n}}{\pi}L=\frac{\sqrt{n}}{\pi}\!\int_0^1e^{-nx^2\!/6}\psi^6(x)\frac{d\psi}{dx}(x)dx\mbox{.}
\end{equation}

Recall that $\psi$ is a continuously differentiable function of $x$ in the interval $[0,1]$. By Proposition \ref{DSH.sec.3.prop.1}, when $x$ belongs to that interval,
$$
\biggl|\frac{d\psi}{dx}(x)\biggr|\leq1.00001
$$
and in turn, as $\psi(0)$ is equal to $0$,
$$
|\psi(x)|\leq1.00001x
$$

As a consequence one obtains, by bounding $\psi$ and its derivative in the integrand of the right-hand side of (\ref{DSH.sec.5.lem.2.eq.2}), that
$$
\frac{\sqrt{n}}{\pi}L\leq\frac{1.00001^7\sqrt{n}}{\pi}\!\int_0^1e^{-nx^2\!/6}x^6dx\mbox{.}
$$

In turn, according to (\ref{DSH.sec.1.thm.0.eq.1.375}) and (\ref{DSH.sec.1.thm.0.eq.1.4}),
$$
\frac{\sqrt{n}}{\pi}L\leq120\frac{1.00001^7}{\sqrt{\pi}n^3}\sqrt{\frac{3}{2}}^7\mbox{.}
$$

Observe that the coefficient of $1/n^3$ in the right-hand side of this inequality is less than $280$. Therefore, combining it with (\ref{DSH.sec.5.lem.2.eq.1.5}) yields
\begin{equation}\label{DSH.sec.5.lem.2.eq.1.7}
\Biggl|\frac{2\sqrt{n}}{\pi}\!\int_0^{\psi(1)}\biggl(\frac{\sin\,u}{u}\biggr)^{\!n-2}\!\cos\Bigl(2\sqrt{d}tu\Bigr)f(u)du+\frac{2}{45}F_{n,4}(t)\Biggr|\leq\frac{3872}{225n^3}
\end{equation}
and combining it with (\ref{DSH.sec.5.lem.2.eq.1.6}) yields
\begin{equation}\label{DSH.sec.5.lem.2.eq.1.8}
\Biggl|\frac{2\sqrt{n}}{\pi}\!\int_0^{\psi(1)}\biggl(\frac{\sin\,u}{u}\biggr)^{\!n-1}\!\cos\Bigl(2\sqrt{d}tu\Bigr)g(u)du+\frac{1}{45}F_{n,4}(t)\Biggr|\leq\frac{32}{9n^3}\mbox{.}
\end{equation}

It follows from (\ref{DSH.sec.5.eq.1}) that $|f(u)|$ is at most $4$ and that
$$
\biggl|\frac{\sin u}{u}g(u)\biggr|\leq\frac{7}{3}\mbox{.}
$$
Hence, by the triangle inequality,
\begin{equation}\label{DSH.sec.5.lem.2.eq.4}
\Biggl|\frac{2\sqrt{n}}{\pi}\!\int_{\psi(1)}^{+\infty}\biggl(\frac{\sin\,u}{u}\biggr)^{\!n-2}\!\cos\Bigl(2\sqrt{d}tu\Bigr)f(u)du\Biggr|\leq\frac{8\sqrt{n}}{\pi}\!\int_{\psi(1)}^{+\infty}\biggl|\frac{\sin\,u}{u}\biggr|^{n-2}du
\end{equation}
and
\begin{equation}\label{DSH.sec.5.lem.2.eq.4.5}
\Biggl|\frac{2\sqrt{n}}{\pi}\!\int_{\psi(1)}^{+\infty}\biggl(\frac{\sin\,u}{u}\biggr)^{\!n-1}\!\cos\Bigl(2\sqrt{d}tu\Bigr)g(u)du\Biggr|\leq\frac{14\sqrt{n}}{3\pi}\!\int_{\psi(1)}^{+\infty}\biggl|\frac{\sin\,u}{u}\biggr|^{n-2}du\mbox{.}
\end{equation}

However, using the same strategy as for Lemma \ref{DSH.sec.1.lem.0} and by (\ref{DSH.sec.1.eq.4.5}),
$$
\int_{\psi(1)}^{+\infty}\biggl|\frac{\sin\,u}{u}\biggr|^{n-2}du\leq\pi\biggl(\frac{17}{20}\biggr)^{\!n-2}
$$
and as the right-hand side of this inequality is at most $\pi/(8n^{7/2})$ when $n$ is at least $118$, it follows from (\ref{DSH.sec.5.lem.2.eq.4}), (\ref{DSH.sec.5.lem.2.eq.4.5}), from the triangle inequality that
\begin{equation}\label{DSH.sec.5.lem.2.eq.5}
\Biggl|\frac{2\sqrt{n}}{\pi}\!\int_{\psi(1)}^{+\infty}\biggl(\frac{\sin\,u}{u}\biggr)^{\!n-2}\!\cos\Bigl(2\sqrt{d}tu\Bigr)f(u)du\Biggr|\leq\frac{1}{n^3}\mbox{.}
\end{equation}
and
\begin{equation}\label{DSH.sec.5.lem.2.eq.5.5}
\Biggl|\frac{2\sqrt{n}}{\pi}\!\int_{\psi(1)}^{+\infty}\biggl(\frac{\sin\,u}{u}\biggr)^{\!n-1}\!\cos\Bigl(2\sqrt{d}tu\Bigr)g(u)du\Biggr|\leq\frac{1}{n^3}\mbox{.}
\end{equation}
for any such value of $n$. Now observe that by Lemma \ref{DSH.sec.4.lem.0}, when $n$ is at least $6$,
$$
\biggl|\frac{1}{45}I_{n,4}(t)-\frac{1}{45}F_{n,4}(t)\biggr|\leq\frac{2\sqrt{n}}{45}\biggl(\frac{\pi}{2}\biggr)^{\!4}\biggl(\frac{17}{20}\biggr)^{\!n}
$$
and the right-hand side is at most $1/n^3$ when $n$ is at least $89$. Therefore,
\begin{equation}\label{DSH.sec.5.lem.2.eq.5.1}
\biggl|\frac{1}{45}I_{n,4}(t)-\frac{1}{45}F_{n,4}(t)\biggr|\leq\frac{1}{n^3}
\end{equation}
for any such value of $n$. Combining (\ref{DSH.sec.5.lem.2.eq.1.7}), (\ref{DSH.sec.5.lem.2.eq.5}), and (\ref{DSH.sec.5.lem.2.eq.5.1}) using the triangle inequality shows that when $n$ is at least $118$,
$$
\Biggl|\frac{2\sqrt{n}}{\pi}\!\int_0^{+\infty}\biggl(\frac{\sin\,u}{u}\biggr)^{\!n-2}\!\cos\Bigl(2\sqrt{d}tu\Bigr)f(u)du+\frac{2}{45}I_{n,4}(t)\Biggr|\leq\biggl(\frac{3872}{225}+3\biggr)\frac{1}{n^3}\mbox{.}
$$

Since the coefficient of $1/n^3$ in the right-hand side is less than $21$, this establishes (\ref{DSH.sec.5.lem.2.eq.6}). Finally, combining (\ref{DSH.sec.5.lem.2.eq.1.8}), (\ref{DSH.sec.5.lem.2.eq.5.5}) and (\ref{DSH.sec.5.lem.2.eq.5.1}) by means of the triangle inequality shows that if $n$ is at least $118$, then
$$
\Biggl|\frac{2\sqrt{n}}{\pi}\!\int_0^{+\infty}\biggl(\frac{\sin\,u}{u}\biggr)^{\!n-1}\!\cos\Bigl(2\sqrt{d}tu\Bigr)g(u)du+\frac{1}{45}I_{n,4}(t)\Biggr|\leq\biggl(\frac{32}{9}+2\biggr)\frac{1}{n^3}\mbox{.}
$$

Since $32/9$ is less than $4$, this completes the proof.
\end{proof}

The following lemma provides the sign of (\ref{DSH.sec.5.thm.1.eq.1}) when $n$ is larger than an explicit function of $t$. Recall that the two numbers $\gamma^-$ and $\gamma^+$ appearing in the statement of this lemma are the non-negative values of $t$ such that $1-24t^2+48t^4$ vanishes. An explicit expression for them is given by (\ref{DSH.sec.0.eq.0}).

\begin{lem}\label{DSH.sec.5.lem.4}
Consider a non-negative number $t$ other than $\gamma^-$ and $\gamma^+$. If
\begin{equation}\label{DSH.sec.5.lem.2.eq.1}
n\geq\max\biggl\{124,e^{6t^2}\frac{74+84t}{\bigl|1-24t^2+48t^4\bigr|}\biggr\}\mbox{,}
\end{equation}
then the right-hand side of (\ref{DSH.sec.5.thm.1.eq.1}) is negative when $t<\gamma^-$ or $\gamma^+<t$ and positive when $\gamma^-<t<\gamma^+$ while the right-hand side of (\ref{DSH.sec.5.thm.1.eq.0.5})  is positive when $t<\gamma^-$ or $\gamma^+<t$ and negative when $\gamma^-<t<\gamma^+$.
\end{lem}
\begin{proof}
Assume that $n$ is at least $124$. According to Corollary \ref{DSH.sec.4.cor.1},
$$
\Biggl|\frac{1}{45}I_{n,4}(t)-\frac{27}{45}\sqrt{\frac{6}{\pi}}\frac{1-24t^2+48t^4}{n^2e^{6t^2}}\Biggr|<\frac{2267+3132t}{45n^3}\mbox{.}
$$

Hence, by Lemma \ref{DSH.sec.5.lem.2} and the triangle inequality,
\begin{multline}\label{DSH.sec.5.lem.4.eq.2}
\Biggl|\frac{2\sqrt{n}}{\pi}\!\int_0^{+\infty}\biggl(\frac{\sin\,u}{u}\biggr)^{\!n-2}\!\cos\Bigl(2\sqrt{d}tu\Bigr)f(u)du+\frac{54}{45}\sqrt{\frac{6}{\pi}}\frac{1-24t^2+48t^4}{n^2e^{6t^2}}\Biggr|\\
\hfill<\frac{5479+6264t}{45n^3}
\end{multline}
and
\begin{multline}\label{DSH.sec.5.lem.4.eq.3}
\Biggl|\frac{2\sqrt{n}}{\pi}\!\int_0^{+\infty}\biggl(\frac{\sin\,u}{u}\biggr)^{\!n-2}\!\cos\Bigl(2\sqrt{d}tu\Bigr)g(u)du+\frac{27}{45}\sqrt{\frac{6}{\pi}}\frac{1-24t^2+48t^4}{n^2e^{6t^2}}\Biggr|\\
\hfill<\frac{2537+3132t}{45n^3}
\end{multline}
for any non-negative number $t$. It follows from (\ref{DSH.sec.5.lem.4.eq.2}) that if
\begin{equation}\label{DSH.sec.5.lem.4.eq.4}
54\sqrt{\frac{6}{\pi}}\frac{\bigl|1-24t^2+48t^4\bigr|}{n^2e^{6t^2}}\geq\frac{5479+6264t}{n^3}
\end{equation}
then $1-24t^2+48t^4$ and the right-hand side of (\ref{DSH.sec.5.thm.1.eq.1}) are both non-zero and their signs are opposite. Likewise, it follows from (\ref{DSH.sec.5.lem.4.eq.3}) that if
\begin{equation}\label{DSH.sec.5.lem.4.eq.5}
27\sqrt{\frac{6}{\pi}}\frac{\bigl|1-24t^2+48t^4\bigr|}{n^2e^{6t^2}}\geq\frac{2537+3132t}{n^3}
\end{equation}
then $1-24t^2+48t^4$ and the right-hand side of (\ref{DSH.sec.5.thm.1.eq.0.5}) are both positive or both negative. Now recall that $1-24t^2+48t^4$ is positive precisely when $t<\gamma^-$ or $\gamma^+<t$ and negative precisely when $\gamma^-<t<\gamma^+$. Since
$$
n\geq{e^{6t^2}}\frac{74+84t}{\bigl|1-24t^2+48t^4\bigr|}
$$
implies both (\ref{DSH.sec.5.lem.4.eq.4}) and (\ref{DSH.sec.5.lem.4.eq.5}), this completes the proof.
\end{proof}

Using Lemma \ref{DSH.sec.5.lem.4} jointly with Theorem \ref{DSH.sec.5.thm.1} provides the following.

\begin{thm}\label{DSH.sec.5.thm.2}
Consider a non-negative number $t$ other than $\gamma^-$ and $\gamma^+$. Further consider an integer $n$ at most $d$ satisfying
\begin{equation}\label{DSH.sec.5.thm.2.eq.1}
n\geq\max\biggl\{124,e^{6t^2}\frac{74+84t}{\bigl|1-24t^2+48t^4\bigr|}\biggr\}\mbox{.}
\end{equation}

Let $H$ be a hyperplane of $\mathbb{R}^d$ whose distance to the center of the hypercube $[0,1]^d$ is equal to $t$. When $H$ is orthogonal to a diagonal of a $n$-dimensional face of $[0,1]^d$, the $(d-1)$-dimensional volume of $H\cap[0,1]^d$ is
\begin{enumerate}
\item[(i)] strictly locally maximal if $t<\gamma^-$ or $\gamma^+<t$ and $n$ is equal to $d$,
\item[(ii)] strictly locally minimal if $\gamma^-<t<\gamma^+$ and $n$ is equal to $d$, and
\item[(iii)] never locally extremal if $n$ is less than $d$.
\end{enumerate}
\end{thm}

Theorems \ref{DSH.sec.0.thm.3} and \ref{DSH.sec.0.thm.3.5} stated in the introduction can now be proven.

\begin{proof}[Proof of Theorems \ref{DSH.sec.0.thm.3} and \ref{DSH.sec.0.thm.3.5}]
First observe that the right-hand side of (\ref{DSH.sec.5.thm.2.eq.1}) is a convex function of $t$ over the intervals $[0,\gamma^-[$, $]\gamma^-,\gamma^+[$, and $]\gamma^+,+\infty[$. As a consequence, it suffices to show that (\ref{DSH.sec.5.thm.2.eq.1}) holds when $n$ is large enough and $t$ is equal to $0$, to $\gamma^-\pm{K/n}$, to $\gamma^+\pm{K/n}$, or to $\sqrt{(\log n)/6}$ where $K$ is a constant. If $t$ is equal to $0$, then the right-hand side of (\ref{DSH.sec.5.thm.2.eq.1}) is a constant and the desired result is immediate. When $t$ is equal to $\sqrt{(\log n)/6}$, the result is also immediate because in that case, $e^{6t^2}$ is equal to $n$ and
$$
\lim_{n\rightarrow+\infty}\frac{74+84t}{\bigl|1-24t^2+48t^4\bigr|}=0\mbox{.}
$$

Now observe that, if $t$ is equal to $\gamma^-\pm{K/n}$ where $K$ is a constant, then
$$
\lim_{n\rightarrow+\infty}\frac{e^{6t^2}(74+84t)}{n\bigl|1-24t^2+48t^4\bigr|}=e^{6(\gamma^-)^2}\frac{74+84\gamma^-}{16\sqrt{6}K\gamma^-}
$$
and if $t$ is equal to $\gamma^+\pm{K/n}$, then the equality obtained from this one by substituting $\gamma^+$ for $\gamma^-$ holds. As the coefficients of $1/K$ in the right-hand side of these two equalities are less than $76$, this completes the proof.
\end{proof}

The remainder of the section is devoted to extending Theorem~\ref{DSH.sec.5.thm.2} down to the case when $n$ is equal to $4$ for certain ranges of $t$. As in Section \ref{DSH.sec.3}, this is done by using symbolic computation and by taking advantage of certain piecewise polynomial expressions for the right-hand sides of (\ref{DSH.sec.5.thm.1.eq.1}) and (\ref{DSH.sec.5.thm.1.eq.0.5}) that allow to determine their signs. More precisely, these signs can be derived as follows from Lemmas~3.2 and~4.1 in \cite{Pournin2024} and from Equations (4) and (9) in \cite{Pournin2025}.

\begin{lem}\label{DSH.sec.5.lem.5}
Consider a non-negative number $t$. For every integer $n$ at least $4$, the sign of the right-hand side of (\ref{DSH.sec.5.thm.1.eq.1}) is the same as the sign of
\begin{equation}\label{DSH.sec.5.lem.5.eq.2}
\sum_{i=0}^{\lfloor{z}\rfloor}(-1)^i{n\choose{i}}\biggl(\frac{i(n-i)}{n-1}-\biggl(\frac{n}{2}-i\biggr)\frac{(z-i)}{n-2}+\frac{2n(z-i)^2}{(n-1)(n-2)}\biggr)(z-i)^{n-3}
\end{equation}
and the sign of the right-hand side of (\ref{DSH.sec.5.thm.1.eq.0.5}) is the same as the sign of
\begin{equation}\label{DSH.sec.5.lem.5.eq.1}
\sum_{i=0}^{\lfloor{z}\rfloor}(-1)^i{n\choose{i}}\biggl(\frac{n}{12}-\biggl(\frac{n}{2}-i\biggr)\frac{(z-i)}{n-2}+\frac{n(z-i)^2}{(n-1)(n-2)}\biggr)(z-i)^{n-3}
\end{equation}
where
\begin{equation}\label{DSH.sec.5.lem.5.eq.3}
z=\frac{n}{2}-t\sqrt{n}\mbox{.}
\end{equation}
\end{lem}

Observe that (\ref{DSH.sec.5.lem.5.eq.2}) and (\ref{DSH.sec.5.lem.5.eq.1}) are functions of $t$ by the change of variables (\ref{DSH.sec.5.lem.5.eq.3}) and that both of these functions of $t$ are polynomial in the intervals
\begin{equation}\label{DSH.sec.5.eq.3}
\biggl[\frac{\sqrt{n}}{2}-\frac{i+1}{\sqrt{n}},\frac{\sqrt{n}}{2}-\frac{i}{\sqrt{n}}\biggr]
\end{equation}
where $i$ is any integer satisfying $0\leq{i}\leq{n}$. It therefore follows from Lemma~\ref{DSH.sec.5.lem.5} that computing the signs of the right-hand sides of (\ref{DSH.sec.5.thm.1.eq.1}) and (\ref{DSH.sec.5.thm.1.eq.0.5}) amounts to compute the sign of these polynomial functions at the boundary of these intervals and to estimate their roots within these intervals, which can be done using symbolic computation. For any given dimension $n$, it follows from Theorem~\ref{DSH.sec.5.thm.2} that not all of the intervals of the form (\ref{DSH.sec.5.eq.3}) may need to be considered in the computation. Indeed, recall that the right-hand side of (\ref{DSH.sec.5.thm.2.eq.1}) is a convex function of $t$ in the intervals $[0,\gamma^-[$, $]\gamma^-,\gamma^+[$, and $]\gamma^+,+\infty[$. Hence, for any fixed $n$, if the bounds of (\ref{DSH.sec.5.eq.3}) are both contained in one of these intervals and the value of the right-hand side of (\ref{DSH.sec.5.thm.2.eq.1}) at both of these bounds is at most $n$, then Theorem~\ref{DSH.sec.5.thm.2} provides the result for every number $t$ in the interval and no computation is needed. The signs of (\ref{DSH.sec.5.lem.5.eq.2}) and (\ref{DSH.sec.5.lem.5.eq.1}) have been computed when $4\leq{n}\leq700$ using symbolic computation. Note that these signs have already been studied in \cite{Pournin2024,Pournin2025} when $4\leq{n}\leq300$. This computation results in Theorem \ref{DSH.sec.0.thm.4}, stated in the introduction.

Note that in the special case when $t$ is equal to $0$ and $H$ is orthogonal to a diagonal of a proper face of $[0,1]^d$ of dimension exactly $4$, Assertion (iii) in the statement of Theorem \ref{DSH.sec.0.thm.4} is obtained from~\cite[Theorem 1.4]{AmbrusGargyan2024b} rather than the described computation because (\ref{DSH.sec.5.lem.5.eq.1}) and therefore (\ref{DSH.sec.5.thm.1.eq.0.5}) vanish when $n$ is equal to $4$ and $t$ to $0$ \cite{Pournin2025}. In that case, Theorem \ref{DSH.sec.5.thm.1} cannot be applied.

The right-hand side of (\ref{DSH.sec.0.thm.4.eq.0}) in the statement of Theorem \ref{DSH.sec.0.thm.4} is about $0.14385$ and its expression can be recovered from \cite[Equation (46)]{Pournin2024}. It is noteworthy that the results stated by Theorem \ref{DSH.sec.0.thm.4} when $t$ gets close to the right-hand side of (\ref{DSH.sec.0.thm.4.eq.0}) in Assertions (i) and (iii) or to $1/4$ in Assertion (iii) cannot be extended beyond these values. Indeed, according to \cite[Proposition 6.1]{Pournin2024}, if $d$ is equal to $4$ then the volume of $H\cap[0,1]^4$ is strictly locally minimal when $H$ is orthogonal to a diagonal of $[0,1]^4$ and $t$ is greater than the right-hand side of (\ref{DSH.sec.0.thm.4.eq.0}) but less than $5/8$. Moreover, according to \cite[Proposition 3.1]{Pournin2025}, if $d$ is greater than $4$ then the $(d-1)$-dimensional volume of $H\cap[0,1]^d$ is strictly locally minimal when $H$ is orthogonal to a diagonal of a $4$-dimensional face of $[0,1]^d$ and $t$ is greater than the right-hand side of~(\ref{DSH.sec.0.thm.4.eq.0}) but less than $1/4$. However, these limiting values for $t$ are all related to the $4$-dimensional case and, considering only dimensions above $4$, the computations further allow to prove the following
\begin{thm}\label{DSH.sec.5.thm.3}
Consider a non-negative number $t$. Let $H$ be a hyperplane of $\mathbb{R}^d$ whose distance to the center of $[0,1]^d$ is equal to $t$. If $d$ is at least $5$, then the $(d-1)$-dimensional volume of $H\cap[0,1]^d$ is
\begin{enumerate}
\item[(i)] strictly locally maximal when $t$ satisfies $0\leq{t}\leq0.19436$ and $H$ is orthogonal to a diagonal of the hypercube $[0,1]^d$ and
\item[(ii)] not locally extremal (even weakly so) when either $0\leq{t}\leq0.19436$ or $0.23593\leq{t}\leq0.59495$ and $H$ is orthogonal to a diagonal of a proper face of dimension at least $5$ of the hypercube $[0,1]^d$.
\end{enumerate}
\end{thm}

Note that the ranges for $t$ stated by Theorem \ref{DSH.sec.5.thm.3} and well as the ranges in Theorem \ref{DSH.sec.0.thm.4} near their bounds other than $1/4$, $1/2$, or the right-hand side of~(\ref{DSH.sec.0.thm.4.eq.0}) can be enlarged by extending the computation to $n$ greater than $700$.

\medskip

\noindent\textbf{Acknowledgement.} The author is grateful to Alexandros Eskenazis for bringing the article of Bartha, Fodor, and Gonz{\'a}lez Merino \cite{BarthaFodorGonzalezMerino2021} to his attention, to Hermann K{\"o}nig for enlightening discussions about hypercube sections, and to Antoine Deza for comments on an early version of this manuscript.

\bibliography{DeepSections}
\bibliographystyle{ijmart}

\end{document}